\documentclass{article}
\usepackage[letterpaper]{geometry}

\usepackage{amsmath}             
\usepackage{amsfonts}
\usepackage{mathrsfs}
\usepackage{amsthm}
\usepackage{amssymb}
\usepackage[capitalize,nameinlink]{cleveref}
\usepackage[table]{xcolor}       
\usepackage{multirow}            
\usepackage{longtable}           
\usepackage[flushleft]{threeparttable}      %
\usepackage{booktabs}            
\usepackage{diagbox}             
\usepackage{calc}                
\usepackage{enumerate}
\usepackage{color}
\usepackage{url}
\usepackage{enumerate}
\usepackage[titletoc,title]{appendix}

\usepackage{tikz}
\usetikzlibrary{calc}
\usetikzlibrary{tikzmark}
\newcommand{\tm}{\tikzmark}

\usepackage{kbordermatrix, etoolbox}
\renewcommand{\kbldelim}{[\hspace*{-\arraycolsep}}
\patchcmd{\kbordermatrix}{\right\kbrdelim}{\hspace*{-\arraycolsep}\right\kbrdelim}{}{}
\newtheorem{theorem}{Theorem}[section]
\newtheorem{corollary}[theorem]{Corollary}
\newtheorem{proposition}[theorem]{Proposition}
\newtheorem{lemma}[theorem]{Lemma}

\newtheorem{definition}[theorem]{Definition}
\newtheorem{example}[theorem]{Example}
\newtheorem{question}[theorem]{Question}

\DeclareMathOperator{\Ker}{Ker}
\DeclareMathOperator{\GF}{GF}
\DeclareMathOperator{\GH}{GH}

\newcommand{\Z}{{\mathbb{Z}}}
\newcommand{\ba}{{\mathbf{a}}}
\newcommand{\bb}{{\mathbf{b}}}
\newcommand{\bc}{{\mathbf{c}}}
\newcommand{\bd}{{\mathbf{d}}}
\newcommand{\be}{{\mathbf{e}}}
\newcommand{\br}{{\mathbf{r}}}
\newcommand{\bz}{{\mathbf{0}}}
\newcommand{\floor}[1]{{\lfloor{#1}\rfloor}}

\newcommand{\DrawBox}[3][]{%
  \tikz[overlay,remember picture]{%
    \coordinate (TopLeft)     at ($(#2)+(-0.1em,0.9em)$);
    \coordinate (BottomRight) at ($(#3)+(0.15em,-0.3em)$);%
    \path (TopLeft); \pgfgetlastxy{\XCoord}{\IgnoreCoord};
    \path (BottomRight); \pgfgetlastxy{\IgnoreCoord}{\YCoord};
    \fill [red,#1] (TopLeft) rectangle (BottomRight);
  }
}

\definecolor{thm1}{gray}{0.88}
\definecolor{thm2}{gray}{0.7}
\definecolor{new}{gray}{0.57}
\definecolor{base}{gray}{0.45}
\definecolor{basetext}{gray}{0}

\long\def\symbolfootnote[#1]#2{\begingroup%
\def\thefootnote{\fnsymbol{footnote}}\footnote[#1]{#2}\endgroup}

\allowdisplaybreaks[4]

\begin{document}

\title{A new structure for difference matrices over abelian $p$-groups}
\author{Koen van Greevenbroek \and Jonathan Jedwab}
\date{8 June 2018 (revised 28 November 2018)}
\maketitle

\symbolfootnote[0]{
Department of Mathematics,
Simon Fraser University, 8888 University Drive, Burnaby BC V5A 1S6, Canada.
\par
J.~Jedwab is supported by NSERC.
\par
Email: {\tt kvangree@sfu.ca}, {\tt jed@sfu.ca}
}

\begin{abstract}
A difference matrix over a group is a discrete structure that is intimately related to many other combinatorial designs, including mutually orthogonal Latin squares, orthogonal arrays, and transversal designs.
Interest in constructing difference matrices over $2$-groups has been renewed by the recent discovery that these matrices can be used to construct large linking systems of difference sets, which in turn provide examples of systems of linked symmetric designs and association schemes. 
We survey the main constructive and nonexistence results for difference matrices, beginning with a classical construction based on the properties of a finite field. We then introduce the concept of a contracted difference matrix, which generates a much larger difference matrix. We show that several of the main constructive results for difference matrices over abelian $p$-groups can be substantially simplified and extended using contracted difference matrices. In particular, we obtain new linking systems of difference sets of size $7$ in infinite families of abelian $2$-groups, whereas previously the largest known size was~$3$.
\end{abstract}

\section{Introduction}
\label{sec:introduction}

  Let $G$ be a non-trivial group. A \emph{$(G, m, \lambda)$ difference matrix over~$G$} is an $m \times \lambda |G|$ matrix $(a_{ij})$ with $0 \le i \le m-1$ and $0 \le j \le \lambda |G|-1$ and each entry $a_{ij} \in G$ such that, for all distinct rows $i$ and $\ell$, the multiset of ``differences''
\[
\{a_{ij}a_{\ell j}^{-1} : 0 \le j \le \lambda |G|-1\}
\]
contains each element of $G$ exactly $\lambda$ times. 
When the group $G$ is abelian, we shall (except in \cref{sec:red-link}) use additive rather than multiplicative group notation.

\begin{example}
  Let $G = \Z_2^3$ and represent the element $(u,v,w) \in G$ in the compressed form $uvw$. The matrix
  \begin{equation*} 
    \left[
    \begin{array}{cccccccc}
      000 & 000 & 000 & 000 & 000 & 000 & 000 & 000 \\
      \rowcolor{thm2} 000 & 101 & 111 & 010 & 110 & 011 & 001 & 100 \\
      000 & 100 & 010 & 110 & 001 & 101 & 011 & 111 \\
      \rowcolor{thm2} 000 & 111 & 110 & 001 & 011 & 100 & 101 & 010 \\
      000 & 001 & 101 & 100 & 111 & 110 & 010 & 011 \\
    \end{array}
    \right]
  \end{equation*}
is a $(G, 5, 1)$ difference matrix. For example, the differences between corresponding entries of the two shaded rows are
\[
000, 010, 001, 011, 101, 111, 100, 110,
\] 
in which each element of $G$ appears exactly once. 
\end{example}

Difference matrices are related to many other combinatorial designs, including mutually orthogonal Latin squares, orthogonal arrays, transversal designs, whist tournaments, generalized Steiner triple systems, and optical orthogonal codes \cite{colbourn-diffmatrices}, \cite{pan-chang}. 
Recently, difference matrices over $2$-groups were used as the key ingredient in a new construction of linking systems of difference sets \cite{jedwab-li-simon-arxiv}.
The central objective is to determine, for a given group $G$ and parameter $\lambda$, the largest number of rows $m$ for which a $(G,m,\lambda)$ difference matrix exists. Colbourn~\cite{colbourn-diffmatrices} gives a concise summary of known existence and nonexistence results as of 2007.

This paper is organized so that the material up to the end of \cref{sec:red-link} is a survey, whereas that from Section~\ref{sec:contr-diff-matrices} onwards presents new ideas and results.
In \cref{sec:basic} we describe some basic properties of difference matrices, including nonexistence results and connections to other combinatorial structures.
In \cref{sec:diff-matrices} we review the major constructions for difference matrices, principally: a classical construction over elementary abelian $p$-groups based on finite fields; a composition construction based on the Kronecker product; and a construction of 4-row difference matrices over abelian noncyclic groups. In \cref{sec:red-link} we explain how difference matrices with $\lambda=1$ over certain $2$-groups were recently used to construct linking systems of difference sets. 

In \cref{sec:contr-diff-matrices} we introduce the concept of a contracted difference matrix over an abelian $p$-group, which generates a much larger difference matrix over the same group. We derive a finite field construction, a composition construction, and an abelian noncyclic 2-group construction for contracted difference matrices. These constructions are significantly simpler and more compact than the corresponding constructions for difference matrices given in \cref{sec:diff-matrices}, but we show that they can often be used to produce results for difference matrices that are just as powerful as those obtained in \cref{sec:diff-matrices}.
In \cref{sec:new-contr-diff} we present four examples of contracted difference matrices found by computer search. These examples generate new infinite families of (contracted) difference matrices over abelian $2$-groups with more rows than previously known, from which we in turn construct larger linking systems of difference sets than previously known.
In \cref{sec:open-questions} we present some open questions about (contracted) difference matrices.
Appendix~\ref{sec:list-contr-diff} contains an example of the largest known contracted difference matrix over each abelian $2$-group of order at most~$64$.
Python 3 code for checking and searching for (contracted) difference matrices is available at https://gitlab.com/koenvg/contracted-difference-matrices.

\section{Basic properties}
\label{sec:basic}

In this section we present some basic properties of difference matrices, including nonexistence results and connections to other combinatorial designs.

If $G$ is a group and $A = (a_{ij})$ is a $(G,m,\lambda)$ difference matrix, then the difference matrix property is preserved when each entry of a column of $A$ is right-multiplied by a fixed $g \in G$, because $(a_{ij}g)(a_{\ell j}g)^{-1}= a_{ij} a_{\ell j}^{-1}$. By right-multiplying all entries of each column $j$ of $A$ by $a_{0j}^{-1}$, we may therefore assume that each entry of row $0$ of $A$ is~$1_G$. 
The difference property of the matrix then implies that, for each $i \ge 1$, row $i$ of $G$ contains every element of $G$ exactly $\lambda$ times. 
We may likewise right-multiply all entries of each row $i$ by $a_{i0}^{-1}$, so that each entry of column $0$ of $A$ is also~$1_G$. The resulting matrix is in \emph{normalized form}.

If a $(G,m,\lambda)$ difference matrix with $m \ge 2$ exists, then deleting one row gives a $(G,m-1,\lambda)$ difference matrix.

The existence of a $(G,m,\lambda)$ difference matrix implies the existence of a resolvable orthogonal array ${\rm OA}_\lambda(m,|G|)$ and a transversal design ${\rm TD}_\lambda(m+1,|G|)$, and is a generalized Bhaskar Rao design ${\rm GBRD}(m,m,\lambda|G|;G)$ \cite{colbourn-diffmatrices}.

A trivial $(G,2,\lambda)$ difference matrix exists for every group $G$ and every integer $\lambda \ge 1$, for example comprising a first row containing $\lambda |G|$ copies of the identity~$1_G$ and a second row containing each element of $G$ exactly $\lambda$ times.
The number of rows $m$ in a nontrivial $(G,m,\lambda)$ difference matrix therefore satisfies $m \ge 3$, and by the following counting result it also satisfies $m \le \lambda|G|$.
\begin{theorem}[Jungnickel~{\cite[Theorem~2.2]{jungnickel-diffmatrices}}] \label{thm:jungnickel-nonexistence}
Let $G$ be a group, and suppose there exists a $(G,m,\lambda)$ difference matrix. Then $m \le \lambda|G|$.
\end{theorem}
\begin{proof}
The following self-contained argument is adapted from the proof of a more general result given in \cite[Proposition~3.1]{jungnickel-diffmatrices}.
By assumption there exists a $(G,m,\lambda)$ difference matrix $A= (a_{ij})$, and we may assume that $A$ is in normalized form. 
For $g \in G$ and $0 \le j \le \lambda|G|-1$, let $n_{gj}$ be the number of times $g$ occurs in column $j$ of $A$, so that
\[
n_{gj}=\sum_{i=0}^{m-1} I[a_{ij}=g]
\]
where $I[X]$ is the indicator function of event~$X$. Then
\begin{align}
\sum_{j>0} \sum_{g \in G} n_{gj} 
 &= \sum_{j>0} \sum_i \sum_{g \in G} I[a_{ij}=g] \nonumber \\
 &= \sum_{j>0} \sum_i 1 \nonumber \\
 &= (\lambda|G|-1)m \label{eqn:ngj},
\end{align}
whereas
\begin{align*}
\sum_{j>0} \sum_{g \in G} n_{gj}^2
 &= \sum_{j>0} \sum_{g \in G} \sum_{i, \ell} I[a_{ij}=a_{\ell j} =g] \\
 &= \sum_{j>0} \sum_{g \in G} 
    \bigg ( \sum_i I[a_{ij}=g] + \sum_{i \ne \ell} I[a_{ij}=a_{\ell j} =g] \bigg ) \\
 &= \sum_{j>0} \sum_i \sum_{g \in G} I[a_{ij}=g] 
  + \sum_{i \ne \ell} \sum_{j>0} \sum_{g \in G} I[a_{ij}=a_{\ell j} =g] \\
 &= \sum_{j>0} \sum_i 1 + \sum_{i \ne \ell} \sum_{j>0} I[a_{ij}=a_{\ell j}] \\
 &= \sum_{j>0} \sum_i 1 + \sum_{i \ne \ell} (\lambda-1)
\end{align*}
because the normalization of $A$ gives $a_{i0} = a_{\ell 0}$ for all distinct $i$ and $\ell$. Therefore
\begin{equation} \label{eqn:ngj2}
\sum_{j>0} \sum_{g \in G} n_{gj}^2 = (\lambda|G|-1)m + m(m-1)(\lambda-1),
\end{equation}
and the result follows by substituting \eqref{eqn:ngj} and \eqref{eqn:ngj2} into the Cauchy-Schwarz inequality
\[
\Big( \sum_{j>0} \sum_{g \in G} n_{gj}\Big)^2 \le (\lambda|G|-1) |G| \sum_{j>0} \sum_{g \in G} n_{gj}^2
\]
and simplifying.
\end{proof}

If the upper bound $m = \lambda|G|$ in \cref{thm:jungnickel-nonexistence} is attained, then the resulting $(G,\lambda|G|,\lambda)$ difference matrix is a square matrix known as a \emph{generalized Hadamard matrix $\GH(|G|,\lambda)$ over $G$} (see \cite{delauney-hcd} for a survey).
In particular, a $\GH(2,2\lambda)$ over the group $(\{-1,1\},\cdot)$ is a \emph{Hadamard matrix} of order~$4\lambda$ (see \cite{horadam-book} or \cite{craigen-hcd}, for example, for background on this much-studied topic). In all known examples of a $\GH(|G|,\lambda)$ over $G$, the group order $|G|$ is a prime power and, if $G$ is not elementary abelian, then $|G|$ is a square \cite[p.~303]{delauney-hcd}. 

We shall be mostly concerned with $(G,m,\lambda)$ difference matrices for which $m < \lambda|G|$, and especially those with $\lambda = 1$ because of several connections to other combinatorial objects. In particular, a $(G,m,1)$ difference matrix is equivalent to a $G$-regular set of $m-1$ mutually orthogonal Latin squares of order $|G|$ \cite[Theorem~1]{jungnickel-latin}, and to a set of $m-2$ pairwise orthogonal orthomorphisms of~$G$ \cite[p.~195]{evans-jcd}. Moreover, a crucial ingredient in a recent construction of reduced linking systems of difference sets \cite{jedwab-li-simon-arxiv} is a $(G,m,1)$ difference matrix for certain $2$-groups $G$, as described in \cref{sec:red-link}.
We shall therefore pay special attention to $(G,m,1)$ difference matrices over $2$-groups. 
A further reason for regarding the case $\lambda=1$ as fundamental is that there are many methods for composing two difference matrices (the composition constructions of Theorems~\ref{thm:prod-jungnickel}, \ref{comp}, \ref{mm'}, and \ref{lambda+mu}), and for constructing a new difference matrix from another (the homomorphism construction of \cref{group-hom}), under all of which the value of $\lambda$ increases or remains the same; in particular, we can use \cref{lambda+mu} to produce a $(G,m,\lambda)$ difference matrix for each $\lambda > 1$ from a $(G,m,1)$ matrix.

The following nonexistence result rules out, as a special case, the existence of a $(G,3,1)$ difference matrix when $G$ is a cyclic $2$-group.
Indeed, we shall see (for example in \cref{thm:rec} and from \cref{small-2-grp-info}) that the currently known existence pattern for $(G,m,1)$ difference matrices over a $2$-group~$G$ of fixed order and for fixed $m$ appears to favor groups of smaller exponent and larger rank.

\begin{theorem}[Hall and Paige~{\cite[Theorem~5]{hall-paige}}, Drake~{\cite[Theorem~1.10]{drake}}] \label{thm:drake-nonexistence}
Let $G$ be a group containing a nontrivial cyclic Sylow $2$-subgroup, and let $\lambda$ be odd. Then there does not exist a $(G,3,\lambda)$ difference matrix.
\end{theorem}

\section{Constructions for difference matrices}
\label{sec:diff-matrices}

In this section we describe some of the principal constructive results for difference matrices, especially as they relate to the case $\lambda =1$.
We sometimes omit proofs, or else describe constructions without proving they satisfy the required poperties.

\subsection{Finite field construction (Drake)}
\label{sec:finite-field}

The following construction, based on properties of a finite field, is a foundational example that shows the upper bound of \cref{thm:jungnickel-nonexistence} can be attained for every elementary abelian group.

\begin{proposition}[Drake~{\cite[Proposition~1.5]{drake}}]
  \label{elem-ab-1}
  Let $p$ be prime and let $n$ be a positive integer. Then the additive form of a multiplication table for $\GF(p^n)$ is a $(\Z_{p}^{n}, p^n, 1)$ difference matrix.
\end{proposition}

\begin{example}
  \label{ex:Z22}
  We use \cref{elem-ab-1} to construct a $(\Z_{2}^2, 4, 1)$ difference matrix. Let $\alpha$ be a root of the primitive polynomial $f(x)=x^2+x+1$ in $\Z_2[x]$, and construct $\GF(2^2)$ as $\Z_2[x]/\langle f(x) \rangle$. 
The additive group of $\GF(2^2)$ is $\Z_2^2$, and the multiplication table of $\GF(2^{2})$ written in additive notation gives the \mbox{$(\Z_2^2, 4, 1)$} difference matrix
  \[
    \kbordermatrix{
 \cdot	  & 0   & 1   & \alpha & \alpha^2 \\
 0        & 00 & 00 & 00 & 00 \\
 1        & 00 & 01 & 10 & 11 \\
 \alpha   & 00 & 10 & 11 & 01 \\
 \alpha^2 & 00 & 11 & 01 & 10 \\
    }.
  \]
\end{example}

\begin{example}
  \label{ex:Z23}
  We use \cref{elem-ab-1} to construct a $(\Z_{2}^3, 8, 1)$ difference matrix. Let $\alpha$ be a root of the primitive polynomial $f(x)=x^3+x+1$ in $\Z_2[x]$, and construct $\GF(2^3)$ as $\Z_2[x]/\langle f(x) \rangle$. 
The additive group of $\GF(2^3)$ is $\Z_2^3$, and the multiplication table of $\GF(2^{3})$ written in additive notation gives the \mbox{$(\Z_2^3, 8, 1)$} difference matrix
  \begin{equation*}
    \DrawBox[fill=thm2]{pic cs:l1}{pic cs:r1}
    \kbordermatrix{
 \cdot	  & 0   & 1   & \alpha & \alpha^2 & \alpha^3 & \alpha^4 & \alpha^5 & \alpha^6 \\
 0        & 000 & 000 & 000 & 000 & 000 & 000 & 000 & 000 \\
 1        & 000 & \tm{l1}001 & 010 & 100 & 011 & 110 & 111 & 101 \\
 \alpha   & 000 & 010 & 100 & 011 & 110 & 111 & 101 & 001 \\
 \alpha^2 & 000 & 100 & 011 & 110\tm{r1} & 111 & 101 & 001 & 010 \\
 \alpha^3 & 000 & 011 & 110 & 111 & 101 & 001 & 010 & 100 \\
 \alpha^4 & 000 & 110 & 111 & 101 & 001 & 010 & 100 & 011 \\
 \alpha^5 & 000 & 111 & 101 & 001 & 010 & 100 & 011 & 110 \\
 \alpha^6 & 000 & 101 & 001 & 010 & 100 & 011 & 110 & 111 \\
    }.
  \end{equation*}
(We shall refer to the shaded entries of this multiplication table in Example~\ref{ex:contr-field}.)
\end{example}

The next result extends the construction of \cref{elem-ab-1} to give examples with $\lambda > 1$.

\begin{lemma}[{\cite[Proposition~1.8]{drake}}, {\cite[Proposition~4.4]{jungnickel-diffmatrices}}]
  \label{group-hom}
  Let $G$ and $H$ be groups. Suppose that $\phi \colon G \to H$ is a surjective homomorphism and that $A = (a_{ij})$ is a $(G, m, \lambda)$ difference matrix. Then $\phi(A) = (\phi(a_{ij}))$ is an $(H, m, \lambda |\Ker \phi|)$ difference matrix.
\end{lemma}
\begin{proof}
The difference of two distinct rows of~$A$ contains each element of $G$ exactly $\lambda$ times, so by the First Isomorphism Theorem the difference of two distinct rows of $\phi(A)$ contains each element of $H$ exactly $\lambda |\Ker \phi|$ times.
\end{proof}

\begin{corollary}[{\cite[Corollary~1.9]{drake}}]
  \label{elem-ab}
  Let $p$ be prime, and let $m, n \geq 1$ and $s \geq 0$ be integers.
Then there exists a $(\Z_p^n, m, p^s)$ difference matrix if and only if $m \le p^{n+s}$.
\end{corollary}
\begin{proof}
The condition $m \le p^{n+s}$ is necessary, by \cref{thm:jungnickel-nonexistence}.
To show existence for $m < p^{n+s}$, delete $p^{n+s}-m$ of the rows of the difference matrix for $m=p^{n+s}$. 
It remains to construct a $(\Z_p^n,p^{n+s},p^s)$ difference matrix.
By \cref{elem-ab-1}, there exists a $(\Z_p^{n+s},p^{n+s},1)$ difference matrix. Apply \cref{group-hom} using a surjective homomorphism $\phi: \Z_p^{n+s} \to \Z_p^n$.
\end{proof}

\begin{example}
  We use \cref{group-hom} to construct a $(\Z_2^2, 8, 2)$ difference matrix from the $(\Z_2^3, 8, 1)$ difference matrix of \cref{ex:Z23}. Apply the canonical homomorphism $\Z_2^3 \to \Z_2^2$ to remove the third component of each element of $\Z_2^3$, giving the $(\Z_2^2, 8, 2)$ difference matrix
  \begin{equation*}
    \begin{bmatrix}
      00 & 00 & 00 & 00 & 00 & 00 & 00 & 00 \\
      00 & 00 & 01 & 10 & 01 & 11 & 11 & 10 \\
      00 & 01 & 10 & 01 & 11 & 11 & 10 & 00 \\
      00 & 10 & 01 & 11 & 11 & 10 & 00 & 01 \\
      00 & 01 & 11 & 11 & 10 & 00 & 01 & 10 \\
      00 & 11 & 11 & 10 & 00 & 01 & 10 & 01 \\
      00 & 11 & 10 & 00 & 01 & 10 & 01 & 11 \\
      00 & 10 & 00 & 01 & 10 & 01 & 11 & 11 \\
    \end{bmatrix}.
  \end{equation*}
\end{example}

\subsection{Composition construction (Buratti)}
\label{sec:composition}

The following composition construction combines difference matrices in groups $H$ and $K$ to produce a difference matrix in~$H \times K$. 
We can use this composition to combine difference matrices in groups of prime power order, including those of \cref{sec:abelian-noncyclic}, giving a rich existence pattern. 

\begin{theorem}[Jungnickel~{\cite[Proposition~4.5]{jungnickel-diffmatrices}}] \label{thm:prod-jungnickel}
  Let $H$ and $K$ be groups. Suppose there exists an $(H,m,\lambda)$ difference matrix and a $(K,m,\mu)$ difference matrix. Then there exists a $(H \times K,m,\lambda \mu)$ difference matrix.
\end{theorem}

\cref{thm:prod-jungnickel} occurs as the case $G = H \times K$ of the following more general composition construction, which combines difference matrices in groups $H$ and $G/H$ to produce a difference matrix in~$G$.

\begin{theorem}[Buratti~{\cite[Theorem~2.5~and~Corollary~2.6]{buratti}}]
  \label{comp}
  Let $G$ be a group containing a normal subgroup~$H$. Suppose that $A$ is an $(H, m, \lambda)$ difference matrix and that $(b_{ij}H)$ is a $(G/H, m, \mu)$ difference matrix, and let $B = (b_{ij})$. Then the matrix each of whose rows is the Kronecker product of the corresponding rows in $A$ and $B$ is a $(G, m, \lambda \mu)$ difference matrix.
\end{theorem}

\begin{example}
  We use \cref{comp} to construct a $(G, 4, 1)$ difference matrix for 
$G = \Z_{4} \times \Z_{2} \times \Z_{2}$, using additive notation. Let $H = \langle 010, 200 \rangle$, so that $G / H = \langle 001+H, 100+H \rangle$ and both $H$ and $G / H$ are isomorphic to $\Z_{2} \times \Z_{2}$. The $(\Z_2^2,2,1)$ difference matrix of \cref{ex:Z22} gives the $(H, 4, 1)$ difference matrix
  \begin{equation*}
    A =
    \begin{bmatrix}
      000 & 000 & 000 & 000 \\
      000 & 010 & 200 & 210 \\
      000 & 200 & 210 & 010 \\
      000 & 210 & 010 & 200
    \end{bmatrix}
  \end{equation*}
  and the $(G / H, 4, 1)$ difference matrix $(b_{ij} + H)$, where
  \begin{equation*}
    B = (b_{ij}) =
    \begin{bmatrix}
      000 & 000 & 000 & 000 \\
      000 & 001 & 100 & 101 \\
      000 & 100 & 101 & 001 \\
      000 & 101 & 001 & 100
    \end{bmatrix}.
  \end{equation*}
  By \cref{comp}, the Kronecker product of corresponding rows in $A$ and $B$ then gives the rows of the $(G, 4, 1)$ difference matrix 
  \begin{equation*}
    \begin{bmatrix}
      (000 + 000) & (000 + 000) & (000 + 000) & (000 + 000) & (000 + 000) & (000 + 000) & \cdots \\
      (000 + 000) & (000 + 001) & (000 + 100) & (000 + 101) & (010 + 000) & (010 + 001) & \cdots \\
      (000 + 000) & (000 + 100) & (000 + 101) & (000 + 001) & (200 + 000) & (200 + 100) & \cdots \\
      (000 + 000) & (000 + 101) & (000 + 001) & (000 + 100) & (210 + 000) & (210 + 101) & \cdots
    \end{bmatrix}
  \end{equation*}
  \begin{equation*}
    =
    \setlength{\arraycolsep}{2.8pt}
    \begin{bmatrix}
      000 & 000 & 000 & 000 & 000 & 000 & 000 & 000 & 000 & 000 & 000 & 000 & 000 & 000 & 000 & 000 \\
      000 & 001 & 100 & 101 & 010 & 011 & 110 & 111 & 200 & 201 & 300 & 301 & 210 & 211 & 310 & 311 \\
      000 & 100 & 101 & 001 & 200 & 300 & 301 & 201 & 210 & 310 & 311 & 211 & 010 & 110 & 111 & 011 \\
      000 & 101 & 001 & 100 & 210 & 311 & 211 & 310 & 010 & 111 & 011 & 110 & 200 & 301 & 201 & 300
    \end{bmatrix}.
  \end{equation*}
\end{example}

Repeated application of the composition construction of \cref{comp} to difference matrices over elementary abelian groups, as given by~\cref{elem-ab-1}, produces examples over a larger set of groups. 

\begin{example}\label{ex:2chains}
  We use \cref{comp} to construct $(G,4,1)$ difference matrix for $G = \Z_{8} \times \Z_{8} \times \Z_{4} \times \Z_{2}$. Form a chain of subgroups
\[
G \supset G_1 \supset G_2,
\]
where $G_1 \cong \Z_4 \times \Z_4 \times \Z_2$ and $G_2 \cong \Z_2 \times \Z_2$ such that $G/G_1 \cong \Z_2^4$ and $G_1/G_2 \cong \Z_2^3$.
By \cref{elem-ab-1}, there is a $(G/G_1,2^4,1)$ and a $(G_1/G_2,2^3,1)$ and a $(G_2,2^2,1)$ difference matrix. 
Use \cref{comp} to combine the first $2^2$ rows of the $(G_2,2^2,1)$ and $(G_1/G_2,2^3,1)$ difference matrices to give a $(G_1,2^2,1)$ difference matrix; then combine this with the first $2^2$ rows of the $(G/G_1,2^4,1)$ difference matrix to give a $(G,2^2,1)$ difference matrix. The number of rows in this difference matrix is $\min(2^4,2^3,2^2) = 2^2$.

However, by using a different chain of subgroups we can instead obtain a $(G,8,1)$ difference matrix: choose
$G'_1 \cong \Z_4 \times \Z_4 \times \Z_2 \times \Z_2$ and $G'_2 \cong \Z_2 \times \Z_2 \times \Z_2$ such that $G/G'_1 \cong \Z_2^3$ and $G'_1/G'_2 \cong \Z_2^3$.
Since each of $G'_2$, $G'_1/G'_2$, and $G/G'_1$ is isomorphic to $\Z_2^3$, combination under \cref{comp} produces a $(G,2^3,1)$ difference matrix. The number of rows in this difference matrix is $\min(2^3,2^3,2^3) =2^3$.
\end{example}

\cref{ex:2chains} shows that when \cref{elem-ab-1} and \cref{comp} are used to produce a difference matrix by choosing a chain of subgroups, some choices can result in a larger number of rows for the final difference matrix than others. 
\cref{prop:buratti-chain} shows how to choose a chain of subgroups that will produce the largest number of rows in the final difference matrix, and \cref{thm:rec} gives the result of making this choice.

\begin{proposition}[Buratti~{\cite[Lemma~2.10]{buratti}}]\label{prop:buratti-chain}
Let $p$ be prime, and let $G$ be an abelian group of order~$p^n$ and exponent $p^e$. Then $\floor{n/e}$ is the largest integer $s$ for which there is a chain of subgroups $G_i$ of $G$ satisfying
\begin{equation*}
G = G_0 \supset G_1 \supset \dots \supset G_r 
\end{equation*}
for some integer $r \ge 0$, such that $G_r$ and each of the quotient groups $G_{i-1}/G_i$ is elementary abelian and of order at least~$p^s$. 
The upper bound $\floor{n/e}$ is attained when each of the $\floor{n/e}$ largest direct factors of $G_{i-1}$ is reduced by a factor of $p$ in $G_i$ and when $G_r$ is the first resulting subgroup that is elementary abelian, and in this case we have
\[
\mbox{$G_{i-1}/G_i \cong \Z_p^{\floor{n/e}}$ for each $i$ satisfying $1 \le i \le r$, and $G_{r} \cong \Z_p^\ell$ for some integer $\ell \ge \floor{n/e}$.}
\]
\end{proposition}

\begin{theorem}[Buratti~{\cite[Theorem~2.11]{buratti}}]
  \label{thm:rec}
  Let $p$ be prime, and let $G$ be an abelian group of order $p^{n}$ and exponent~$p^{e}$. Then there exists a $(G, p^{\floor{n/e}}, 1)$ difference matrix.
\end{theorem}

\begin{proof}
By \cref{prop:buratti-chain}, there is an integer $r \ge 0$ and a chain of subgroups 
\[
G = G_0 \supset G_1 \supset \dots \supset G_r 
\]
such that $G_{i-1}/G_i \cong \Z_p^{\floor{n/e}}$ for each $i$ satisfying $1 \le i \le r$, and $G_r \cong \Z_p^\ell$ for some integer $\ell \ge \floor{n/e}$.
By \cref{elem-ab-1} there is therefore a $(G_{i-1}/G_i,p^\floor{n/e},1)$ difference matrix for each~$i$, and by \cref{elem-ab} there is a $(G_r,p^\floor{n/e},1)$ difference matrix. Apply \cref{comp} to successive pairs $(G,H) = (G_{r-1},G_r), (G_{r-2},G_{r-1}), \dots, (G_0,G_1)$ to obtain a $(G_i, p^\floor{n/e}, 1)$ difference matrix for $i = r-1, r-2, \dots, 0$. The case $i=0$ gives the result.
\end{proof}

\cref{comp} can also be used to obtain the following result.
\begin{theorem}[{\cite[Theorem~2.13]{buratti}}] 
Let $G$ be a group and let $p$ be the smallest prime divisor of $|G|$. Then there exists a $(G,p,1)$ difference matrix.
\end{theorem}

\subsection{Abelian noncyclic construction (Pan and Chang)}
\label{sec:abelian-noncyclic}
By \cref{prop:buratti-chain}, no chain of subgroups will produce a larger number of rows than $p^\floor{n/e}$ in \cref{thm:rec} under combination of \cref{elem-ab-1} and \cref{comp}. However, we now show that larger values are sometimes possible, using the following construction in $2$-groups with large exponent.

\begin{theorem}[Pan and Chang~{\cite[Lemma~3.3]{pan-chang}}]
  \label{thm:pc}
  Let $e$ be a positive integer. Then there exists a $(\Z_{2^{e}} \times \Z_{2}, 4, 1)$ difference matrix.
\end{theorem}

\begin{proof}[Construction for \cref{thm:pc}]
Define sets
  \begin{align*}
    I_{1} &= \{0, 1, \ldots, 2^{e-2} - 1\}, \\
    I_{2} &= \{2^{e-2}, 2^{e-2} + 1, \ldots, 2^{e-1} - 1\}, \\
    I_{1}^{*} &= I_{1} \setminus \{2^{e-2} - 1\} \cup \{2^{e-1} - 1\}, \\
    I_{2}^{*} &= I_{2} \setminus \{2^{e-1} - 1\} \cup \{2^{e-2} - 1\}.
  \end{align*}
  For $0 \le i \le 2^{e-1}-1$, define length 4 column vectors over $\Z_{2^{e}} \times \Z_{2}$ 
  \begin{align*}
    c_{i}^{(0)} &=
                  \begin{cases}
                    \begin{bmatrix} (0,0) & (2i, 0) & (4i, 0) & (-2i, 0) \end{bmatrix}^{\intercal} & \mbox{for $i \in I_{1}$}, \\[1ex]
                    \begin{bmatrix} (0,0) & (2i, 0) & (4i, 1) & (-2i, 1) \end{bmatrix}^{\intercal} & \mbox{for $i \in I_{2}$}, 
                  \end{cases} \\
    c_{i}^{(1)} &=
                  \begin{cases}
                    \begin{bmatrix} (0,0) & (2i, 1) & (4i+1, 0) & (-2i-1, 1) \end{bmatrix}^{\intercal} & \mbox{for $i \in I_{1}$}, \\[1ex]
                    \begin{bmatrix} (0,0) & (2i, 1) & (4i+1, 1) & (-2i-1, 0) \end{bmatrix}^{\intercal} & \mbox{for $i \in I_{2}$}, 
                  \end{cases} \\
    c_{i}^{(2)} &=
                  \begin{cases}
                    \begin{bmatrix} (0,0) & (2i+1, 0) & (4i+2, 0) & (-2i-1, 0) \end{bmatrix}^{\intercal} & \mbox{for $i \in I_{1}$}, \\[1ex]
                    \begin{bmatrix} (0,0) & (2i+1, 0) & (4i+2, 1) & (-2i-1, 1) \end{bmatrix}^{\intercal} & \mbox{for $i \in I_{2}$},
                  \end{cases} \\
    c_{i}^{(3)} &=
                  \begin{cases}
                    \begin{bmatrix} (0,0) & (2i+1, 1) & (4i+3, 0) & (-2i-2, 1) \end{bmatrix}^{\intercal} & \mbox{for $i \in I_{1}^{*}$}, \\[1ex]
                    \begin{bmatrix} (0,0) & (2i+1, 1) & (4i+3, 1) & (-2i-2, 0) \end{bmatrix}^{\intercal} & \mbox{for $i \in I_{2}^{*}$}.
                  \end{cases}
  \end{align*}
  Define a $4 \times 2^{e-1}$ matrix
  \begin{equation*}
    D_{r} = \begin{bmatrix} c_{0}^{(r)} & c_{1}^{(r)} & \ldots & c_{2^{e-1}-1}^{(r)} \end{bmatrix} \quad \mbox{for $r = 0, 1, 2, 3$}.
  \end{equation*}
  Then a $(\Z_{2^e} \times \Z_2,4,1)$ difference matrix is
  \begin{equation*}
    D = \begin{bmatrix} D_{0} \mid D_{1} \mid D_{2} \mid D_{3} \end{bmatrix}.\qedhere
  \end{equation*}
\end{proof}

\begin{example}
  We use the construction for Theorem~\ref{thm:pc} to produce a $(\Z_{8} \times \Z_{2}, 4, 1)$ difference matrix. Set $I_{1} = \{0, 1\}$, $I_{2} = \{2, 3\}$, $I_{1}^{*} = \{0, 3\}$ $I_{2}^{*} = \{1,2\}$. Each $D_r$ is a $4 \times 4$ matrix whose columns are $c_0^{(r)}, c_1^{(r)}, c_2^{(r)}, c_3^{(r)}$, and the constructed matrix is
  \begin{equation*}
    D =
    \begin{bmatrix}
      00 & 00 & 00 & 00 & 00 & 00 & 00 & 00 & 00 & 00 & 00 & 00 & 00 & 00 & 00 & 00 \\
      00 & 20 & 40 & 60 & 01 & 21 & 41 & 61 & 10 & 30 & 50 & 70 & 11 & 31 & 51 & 71 \\
      00 & 40 & 01 & 41 & 10 & 50 & 11 & 51 & 20 & 60 & 21 & 61 & 30 & 71 & 31 & 70 \\
      00 & 60 & 41 & 21 & 71 & 51 & 30 & 10 & 70 & 50 & 31 & 11 & 61 & 40 & 20 & 01
    \end{bmatrix}.
  \end{equation*}
\end{example}

By combining the result of \cref{thm:pc} with a suitable chain of subgroups, we can construct a $(G,4,1)$ difference matrix in all abelian noncyclic $2$-groups.

\begin{theorem}[Pan and Chang~{\cite[Lemma~3.4]{pan-chang}}]
  \label{thm:non-cyclic-2}
  Let $G$ be an abelian noncyclic $2$-group. Then there exists a $(G, 4, 1)$ difference matrix.
\end{theorem}

\begin{proof}
Let $G$ have order $2^n$. The proof is by induction on $n \ge 2$.
The base case $n=2$ requires a $(\Z_2^2, 4, 1)$ difference matrix, which is provided by \cref{elem-ab-1}.
Now assume all cases up to $n-1 \ge 2$ are true. 
If $G = \Z_{2^{n-1}} \times \Z_2$, then case $n$ is true by \cref{thm:pc}, and if $G = \Z_2^3$ then case $n$ is true by \cref{elem-ab}. Otherwise we can choose a noncyclic subgroup $H$ of $G$ such that $G/H \cong \Z_2^2$. There exists a $(G/H,4,1)$ difference matrix by \cref{elem-ab-1}, and an $(H,4,1)$ difference matrix by the inductive hypothesis. Apply \cref{comp} to produce a $(G, 4, 1)$ difference matrix, proving case $n$ and completing the induction. 
\end{proof}

Pan and Chang provide a generalization of \cref{thm:non-cyclic-2} to non-$2$-groups, and a corresponding result for the case $\lambda > 1$.

\begin{theorem}[{\cite[Theorem~1.2]{pan-chang}}]
  \label{pan-chang-non-2}
Let $G$ be an abelian noncyclic group whose Sylow $2$-subgroup of $G$ is trivial or noncyclic. Then there exists a $(G,4,1)$ difference matrix.
\end{theorem}

\begin{theorem}[{\cite[Theorem~1.3]{pan-chang}}]
Let $G$ be an abelian group and let $\lambda > 1$ be an integer.
If $\lambda$ is even, or if $\lambda$ is odd and the Sylow $2$-subgroup of $G$ is trivial or noncyclic, then there exists a $(G,4,\lambda)$ difference matrix.
\end{theorem}

\subsection{Other constructions}
\label{sec:other}

Theorems~\ref{thm:drake-nonexistence} and~\ref{pan-chang-non-2} settle the existence question for a $(G,4,1)$ difference matrix over all abelian groups~$G$, except those that are cyclic of odd order. The following result concerns these groups.

\begin{theorem}[Ge~{\cite[Theorem~3.12]{ge-diffmatrices}}] \label{G41-cyclic} 
Let $v \ge 5$ be an odd integer for which $\gcd(v,27) \ne 9$. Then there exists a $(\Z_v,4,1)$ difference matrix. 
\end{theorem}

The existence pattern for the cases not handled by \cref{G41-cyclic} (namely those for which $\gcd(v,27) = 9$) is not yet clear; it is known {\cite[Lemma~2.2]{pan-chang}} that there does not exist a $(\Z_9,4,1)$ difference matrix. 

The following construction, like that of \cref{elem-ab-1}, is based on properties of a finite field and provides examples of generalized Hadamard matrices. 

\begin{theorem}[Jungnickel~{\cite[Theorem~2.4]{jungnickel-diffmatrices}}] 
Let $p$ be an odd prime and let $n$ be a positive integer. Then there exists a $(\Z_p^n,2p^n,2)$ difference matrix.
\end{theorem}

The construction of \cref{comp} composes difference matrices over groups $H$ and~$G/H$. In contrast, the construction of \cref{mm'} (based on a Kronecker product) and of~\cref{lambda+mu} (based on concatenation) both compose two difference matrices over the same group. 

\begin{theorem}[Shrikhande~{\cite[Theorem~3]{shrikhande}}] \label{mm'}
Let $G$ be a group. Suppose there exists a $(G,m,\lambda)$ difference matrix and a $(G,m',\mu)$ difference matrix. Then there exists a $(G,mm',\lambda \mu |G|)$ difference matrix.
\end{theorem}

\begin{theorem}[Jungnickel~{\cite[Proposition~4.2]{jungnickel-diffmatrices}}] \label{lambda+mu}
Let $G$ be a group. Suppose there exists a $(G,m,\lambda)$ difference matrix and a $(G,m,\mu)$ difference matrix. Then there exists a $(G, m, \lambda + \mu)$ difference matrix.
\end{theorem}

There are several constructions of difference matrices based on the existence of other types of combinatorial design such as pairwise balanced designs, orthogonal arrays, transversal designs, affine resolvable block designs, rings, difference families, and group complementary pairs \cite{colbourn-kreher}, \cite{jungnickel-diffmatrices}, \cite{delauney-genHadamard}, \cite{delauney-diffmatrices}.

\subsection{Computer search results}
After submitting the original version of this paper, we became aware of the following computer search results for difference matrices in groups of order 16. These results, which were found using the viewpoint of orthogonal orthomorphisms, improve on the constructions of Sections~\ref{sec:finite-field}--\ref{sec:abelian-noncyclic}. 

\begin{proposition}[Lazebnik and Thomason~{\cite[p.1556]{lazebnik-thomason}}]
\label{prop:lazebnik-thomason}
The largest number of rows $m$ for which a $(G, m, 1)$ difference matrix exists is
\begin{enumerate}[(i)]
\item $5$ for $G = \Z_8 \times \Z_2$,
\item $8$ for $G = \Z_4 \times \Z_4$,
\item $8$ for $G = \Z_4 \times \Z_2 \times \Z_2$.
\end{enumerate}
\end{proposition}

\section{Reduced linking systems of difference sets}
\label{sec:red-link}

We use multiplicative notation for groups throughout this section.
\begin{definition}
Let $G$ be a group of order $v$ 
and let $D$ be a subset of $G$ with $k$ elements. Then $D$ is a $(v,k,\lambda,n)$-\emph{difference set in $G$} if the multiset $\{d_1 d_2^{-1}: d_1, d_2 \in D \text{ and } d_1 \ne d_2 \}$ contains every non-identity element of $G$ exactly $\lambda$ times and, by convention, $n=k- \lambda$.
\end{definition}

A difference set in a group $G$ is equivalent to a symmetric design with a regular automorphism group \cite{lander} 
(see \cite{jungnickel-survey} and its updates \cite{jungnickel-survey-update}, \cite{jungnickel-survey-update2}, for example, for background).

\begin{definition}
Let $G$ be a group of order $v$ 
and let $\ell \ge 2$. Suppose $\mathcal{R}=\{D_1, D_2, \cdots, D_\ell \}$ is a collection of size $\ell$ of $(v,k, \lambda,n)$-difference sets in $G$. Then $\mathcal{R}$ is a \emph{reduced $(v,k,\lambda,n; \ell)$-linking system of difference sets in $G$ of size $\ell$} if there are integers $\mu,\nu$ such that for all distinct $i,j$ there is some $(v,k,\lambda,n)$-difference set $D(i,j)$ in $G$ satisfying
\begin{align}\label{LinkingProperty}
\sum_{d_i \in D_i} d_i \sum_{d_j \in D_j} d_j^{-1} = (\mu- \nu) \sum_{d \in D(i,j)} d + \nu \sum_{g \in G} g \quad \text{ in } \Z[G].
\end{align}
\end{definition}

A reduced linking system of difference sets is equivalent to a linking system of difference sets \cite[Proposition~1.7]{jedwab-li-simon-arxiv}, as introduced by Davis, Martin, and Polhill~\cite{davis-martin-polhill}. Such a system gives rise to a system of linked symmetric designs, as introduced by Cameron \cite{cameron-doubly} and studied by Cameron and Seidel~\cite{cameron-seidel}, and is equivalent to a 3-class Q-antipodal cometric association scheme~\cite{vandam}. 

Jedwab, Li, and Simon \cite{jedwab-li-simon-arxiv} recently showed how to construct a reduced linking system of difference sets, based on a difference matrix over a $2$-group having $\lambda=1$.

\begin{theorem}[{\cite[Theorems~1.2~and~5.6]{jedwab-li-simon-arxiv}}]\label{thm:jls}
Let $G$ be a group of order $2^{2d+2}$ which contains a central subgroup $E$ isomorphic to $\Z_2^{d+1}$. Let $m \ge 3$ and suppose there exists a $(G/E, m,1)$-difference matrix. Then there exists a reduced linking system of $(v,k,\lambda,n)$-difference sets in $G$ of size $m-1$, where
\begin{equation} \label{eqn:hadamard}
(v,k,\lambda,n) = (2^{2d+2}, 2^d(2^{d+1}-1), 2^d(2^d-1), 2^{2d}).
\end{equation}
\end{theorem}

\begin{proof}[Construction for \cref{thm:jls}]
Let $s= 2^{d+1}-1$ and let $H_1,H_2, \dots, H_s$ be the subgroups of $G$ corresponding to the hyperplanes ($d$-dimensional subspaces) of $E$ when $E$ is regarded as a vector space of dimension $d+1$ over $\GF$$(2)$.
Let the normalized form (see \cref{sec:basic}) of the $(G/E,m,1)$-difference matrix be $B=(b_{ij}E)$ for $0 \le i \le m-1$ and $0 \le j \le s$.
Choose $e_{ij} \in E$ for each $1 \le i \le m-1$ and $ 1 \le j \le s$ arbitrarily and let
\[
D_i = \bigcup_{j=1}^s b_{ij} e_{ij} H_j \quad \text{ for } 1 \le i \le m-1. 
\]
Then $\{D_1, D_2, \dots, D_{m-1}\}$ is a reduced linking system of $(v,k,\lambda,n)$-difference sets in $G$ of size $m-1$, with $(v,k,\lambda,n)$ as given in~\eqref{eqn:hadamard}.
\end{proof}

\begin{example}[{\cite[Example~5.7]{jedwab-li-simon-arxiv}}]
We use the construction for \cref{thm:jls} to produce a reduced linking system of $(16,6,2,4)$-difference sets in $G = \Z_{4} \times \Z_2 \times \Z_2 = \langle x,y,z \rangle$ of size $3$. Let $E = \langle x^2, z \rangle$, which is isomorphic to~$\Z_2^2$. The subgroups of $G$ corresponding to the hyperplanes of $E$ when $E$ is regarded as a vector space of dimension $2$ over $\GF(2)$ are $H_1 = \langle x^2 \rangle, H_2 = \langle z \rangle, H_3 = \langle x^2 z \rangle$. Using the $(\Z_2^2,4,1)$ difference matrix of \cref{ex:Z22}, the matrix $(b_{ij}E)$ is a $(G/E,4,1)$-difference matrix where
\begin{align*} 
(b_{ij}) = 
\begin{bmatrix}
1_{G}       & 1_{G}      & 1_{G}      & 1_{G} \\
1_{G}       & x          & y          & xy \\
1_{G}       & y          & xy         & x \\
1_{G}       & xy         & x          & y
\end{bmatrix} \text{ for } 0 \le i,j \le 3. 
\end{align*}
Take, for example,
\begin{align*} 
(e_{ij}) = 
\begin{bmatrix}
1_E     & 1_E        & 1_E \\
z       & x^2        & x^2 \\
z       & 1_E        & 1_E
\end{bmatrix} 
\text{ for } 1 \le i,j \le 3.
\end{align*}
Then $\{D_1, D_2, D_3\}$ is a reduced linking system of $(16,6,2,4)$-difference sets in $G$ of size $3$, where
where
\begin{align*}
D_1 &=x  H_1 \cup y  H_2 \cup xy H_3, \\
D_2 &=y z H_1 \cup x^3y H_2 \cup x^3H_3, \\
D_3 &=xyz H_1 \cup x H_2 \cup y H_3.
\end{align*}
\end{example}

The application of \cref{thm:jls} to the difference matrices specified in Theorems~\ref{thm:non-cyclic-2} and~\ref{thm:rec} gives the infinite families of linking systems of difference sets of \cref{cor:jls}~(i) and~(ii), respectively.

\begin{theorem}[{\cite[Corollaries~5.8~and~5.9]{jedwab-li-simon-arxiv}}]\label{cor:jls}
Let $G$ be an abelian group of order $2^{2d+2}$, rank at least $d+1$, and exponent $2^e$.
\begin{enumerate}[(i)]
\item
If $e \le d+1$, then there exists a reduced linking system of $(v,k,\lambda,n)$-difference sets in~$G$ of size~$3$, with $(v,k,\lambda,n)$ as given in~\eqref{eqn:hadamard}.

\item
If $2 \le e \le \frac{d+3}{2}$, then there exists a reduced linking system of $(v,k,\lambda,n)$-difference sets in~$G$ of size $2^{ \left\lfloor \frac{d+1}{e-1} \right\rfloor }-1$, with $(v,k,\lambda,n)$ as given in~\eqref{eqn:hadamard}.
\end{enumerate}
\end{theorem}

\section{Contracted difference matrices}
\label{sec:contr-diff-matrices}

We observe that some difference matrices over an abelian $p$-group have a particularly rich structure, in that their rows can be written as a $\Z_p$-linear combination of a small set of rows, and their columns can likewise be written as a $\Z_p$-linear combination of a small set of columns. In this section we capture this structure by introducing the concept of a contracted difference matrix, and present a series of constructions for contracted difference matrices related to those for difference matrices given in \cref{sec:diff-matrices}.

\begin{definition}
Let $p$ be prime and let $M$ be a $k \times \ell$ matrix over an abelian $p$-group~$(G,+)$. The \emph{$p$-expansion of $M$}, written $f_p(M)$, is a $p^{k} \times p^{\ell}$ matrix over $G$ whose rows are indexed by $\br \in \Z_p^k$ and whose columns are indexed by $\bc \in \Z_p^\ell$. The $(\br, \bc)$ entry of $f_p(M)$ is the vector-matrix-vector product $\br M \bc^{\intercal}$ (in which $\br$ and $\bc$ are regarded as row vectors). 
\end{definition}
The $p$-expansion of a matrix can be represented as the product of matrices, as demonstrated in \cref{ex:expansion}.

\begin{definition}\label{defn:cdm}
Let $p$ be prime and let $G$ be an abelian group of order $p^{n}$. A \mbox{$k \times (n + s)$} matrix~$M$ over $G$ is a \emph{$(G, k, s)$ contracted difference matrix} if  the $p^k \times p^{n+s}$ matrix $f_{p}(M)$ is a $(G, p^{k}, p^{s})$ difference matrix.
\end{definition}

\begin{example}\label{ex:expansion} 
The matrix
   $M = \begin{bmatrix}
          01 & 10 & 20 \\
          21 & 01 & 10
        \end{bmatrix}$
over $G = \Z_{4} \times \Z_{2}$ is a $(G, 2, 0)$ contracted difference matrix because its $2$-expansion
  \begin{align*}
    f_{2}(M) 
    &= \begin{bmatrix}
      0 & 0 \\
      0 & 1 \\
      1 & 0 \\
      1 & 1
    \end{bmatrix}
          \begin{bmatrix}
            01 & 10 & 20 \\
            21 & 01 & 10
          \end{bmatrix}
                      \begin{bmatrix}
                        0 & 0 & 0 & 0 & 1 & 1 & 1 & 1 \\
                        0 & 0 & 1 & 1 & 0 & 0 & 1 & 1 \\
                        0 & 1 & 0 & 1 & 0 & 1 & 0 & 1
                      \end{bmatrix} \\
    &=\;
    \setlength\arraycolsep{3.3pt}
    \kbordermatrix{
         & 000 & 001 & 010 & 011 & 100 & 101 & 110 & 111 \\
      00 &  00 &  00 & 00  &  00 &  00 &  00 &  00 & 00 \\
      01 &  00 &  10 & 01  &  11 &  21 &  31 &  20 & 30 \\
      10 &  00 &  20 & 10  &  30 &  01 &  21 &  11 & 31 \\
      11 &  00 &  30 & 11  &  01 &  20 &  10 &  31 & 21
    }
  \end{align*}
(in which the row and column indexing is shown explicitly) is a $(G, 4, 1)$ difference matrix. 
\end{example}

\begin{example}
The matrix
$M = \begin{bmatrix}
       01 & 10 \\
       10 & 11
     \end{bmatrix}$
over $G = \Z_{3} \times \Z_{3}$ is a $(G, 2, 0)$ contracted difference matrix, because its $3$-expansion
  \begin{align*}
    f_{3}(M) &= 
               {\setlength\arraycolsep{3pt}
                \begin{bmatrix}
                  0 & 0 & 0 & 1 & 1 & 1 & 2 & 2 & 2 \\
                  0 & 1 & 2 & 0 & 1 & 2 & 0 & 1 & 2
                \end{bmatrix}^{\intercal}}
               \begin{bmatrix}
                 01 & 10 \\
                 10 & 11
               \end{bmatrix}
               {\setlength\arraycolsep{3pt}
                \begin{bmatrix}
                  0 & 0 & 0 & 1 & 1 & 1 & 2 & 2 & 2 \\
                  0 & 1 & 2 & 0 & 1 & 2 & 0 & 1 & 2
                \end{bmatrix}} \\
             &=
               \kbordermatrix{
                    & 00 & 01 & 02 & 10 & 11 & 12 & 20 & 21 & 22 \\
                 00 & 00 & 00 & 00 & 00 & 00 & 00 & 00 & 00 & 00 \\
                 01 & 00 & 11 & 22 & 10 & 21 & 02 & 20 & 01 & 12 \\
                 02 & 00 & 22 & 11 & 20 & 12 & 01 & 10 & 02 & 21 \\
                 10 & 00 & 10 & 20 & 01 & 11 & 21 & 02 & 12 & 22 \\
                 11 & 00 & 21 & 12 & 11 & 02 & 20 & 22 & 10 & 01 \\
                 12 & 00 & 02 & 01 & 21 & 20 & 22 & 12 & 11 & 10 \\
                 20 & 00 & 20 & 10 & 02 & 22 & 12 & 01 & 21 & 11 \\
                 21 & 00 & 01 & 02 & 12 & 10 & 11 & 21 & 22 & 20 \\
                 22 & 00 & 12 & 21 & 22 & 01 & 10 & 11 & 20 & 02
               }
  \end{align*}
  is a $(G, 9, 1)$ difference matrix. 
\end{example}

If a $(G,k,s)$ contracted difference matrix with $k \ge 2$ exists, then deleting one row gives a $(G,k-1,s)$ contracted difference matrix.

A trivial $(G,1,s)$ contracted difference matrix exists for every abelian $p$-group $G = \Z_{p^{a_1}} \times \Z_{p^{a_2}} \times \dots \times \Z_{p^{a_r}}$ and every integer $s \ge 0$, for example comprising a single row containing the elements 
$\bigcup_{i=1}^r\{\be_i, p \be_i,p^2 \be_i, \dots,p^{a_{i}-1}\be_i\}$ together with $s$ copies of the identity $0_G$, where $\be_i$ is the vector of length $r$ taking the value 1 in position $i$ and 0 in all other positions. 
For $p$ a prime and $G$ an abelian group of order $p^n$, the number of rows in a nontrivial $(G,k,s)$ contracted difference matrix therefore satisfies $k \ge 2$, and by \cref{thm:jungnickel-nonexistence} and \cref{defn:cdm} it also satisfies $k \le n+s$.

We now give a method for testing whether a given matrix is a $(G,k,s)$ contracted difference matrix without calculating the $p$-expansion of the matrix in full. The method simplifies further in the case $s = 0$, which is of particular interest because a $(G,k,0)$ contracted difference matrix produces a difference matrix with $\lambda = 1$ (having special significance, as discussed in \cref{sec:basic}).

\begin{lemma}
  \label{lem:contr-char}
  Let $p$ be prime and let $G$ be an abelian group of order $p^{n}$.
\begin{enumerate}[(i)]
\item
Let $M$ be a $k \times (n+s)$ matrix over $G$. Then $M$ is a $(G, k, s)$ contracted difference matrix if and only if the set
\[
\mbox{$\big \{\ba M \bc^\intercal : \bc \in \Z_p^{n+s} \big\}$ contains each element of $G$ exactly $p^s$ times}
\]
for all nonzero row vectors $\ba = (a_{i})$ of length $k$, where each $a_{i}$ is an integer satisfying $-p < a_{i} < p$.

\item
Let $M$ be a $k \times n$ matrix over $G$. Then $M$ is a $(G, k, 0)$ contracted difference matrix if and only if
  \begin{equation*}
    \ba M \bb^{\intercal} = 0_{G} \quad \mbox{implies} \quad \ba = \bz \text{ or } \bb = \bz
  \end{equation*}
  for all row vectors $\ba = (a_{i})$ and $\bb = (b_{j})$ of length $k$ and $n$ respectively, where each $a_{i}$ and $b_{j}$ is an integer satisfying $-p < a_{i}, b_{j} < p$.
\end{enumerate}
\end{lemma}

\begin{proof}
\mbox{}
\begin{enumerate}[(i)]
\item
  By definition, $M$ is a $(G,k,s)$ contracted difference matrix if and only if $f_p(M)$ is a $(G,p^k,p^s)$ difference matrix. Since the row of $f_p(M)$ indexed by $\br \in \Z_p^k$ comprises the elements of the set $\big \{\br M \bc^\intercal : \bc \in \Z_p^{n+s} \big\}$, this condition holds if and only if the set $\big \{ (\br_1 - \br_2)M \bc^\intercal : \bc \in \Z_p^{n+s} \big \}$ contains each element of $G$ exactly $p^s$ times, for all distinct $\br_1, \br_2 \in \Z_p^k$. Set $\ba = \br_1-\br_2$ to obtain the result.

\item
Using the case $s=0$ in the proof of part (i), we have that $M$ is a $(G,k,0)$ contracted difference matrix if and only if the set $\big \{ (\br_1 - \br_2)M \bc^\intercal : \bc \in \Z_p^n \big \}$ contains each element of $G$ exactly once, for all distinct $\br_1, \br_2 \in \Z_p^k$. 
Since $G$ has order $p^n$, this condition holds if and only if $(\br_1-\br_2) M (\bc_1-\bc_2)^\intercal \ne 0_G$ for all distinct $\br_1, \br_2 \in \Z_p^k$ and all distinct $\bc_1, \bc_2 \in \Z_p^n$. Set $\ba = \br_1-\br_2$ and $\bb = \bc_1-\bc_2$ to obtain the result.\qedhere
\end{enumerate}
\end{proof}

We now present several constructions of contracted difference matrices over abelian $p$-groups, which are related to the constructions for difference matrices given in \cref{sec:diff-matrices} as set out in \cref{tab:related-constructions}. Each of the contracted difference matrix constructions is more compact and simple than the corresponding difference matrix construction, as can be seen by comparing the examples of this section with those of \cref{sec:diff-matrices}.
The proofs of some of the corresponding pairs of results are similar (particularly Corollary~\ref{elem-ab}/\ref{contr-elem-ab}, Corollary~\ref{thm:rec}/\ref{thm:contr-rec}, and Corollary~\ref{thm:non-cyclic-2}/\ref{thm:contr-non-cyclic-2}); however, the construction proving \cref{thm:contr-pc} is considerably more straightforward than that proving \cref{thm:pc}.
By \cref{defn:cdm}, the major results Theorems~\ref{thm:rec} and~\ref{thm:non-cyclic-2} are direct consequences of Corollaries~\ref{thm:contr-rec} and~\ref{thm:contr-non-cyclic-2}, respectively.

\begin{table}
\begin{center}
  \caption{Related constructions in Sections~\ref{sec:diff-matrices} and~\ref{sec:contr-diff-matrices}.}
  \mbox{}\vspace{0ex}\\
  \label{tab:related-constructions} 
\begin{tabular}{|l|l|l|}
  \toprule 
  		& construction of 	& construction of contracted \\
		& difference matrix	& difference matrix \\
  \midrule
  Finite field	& \cref{elem-ab-1} 	& \cref{contr-elem-ab-1} \\
  		& \cref{elem-ab} 	& \cref{contr-elem-ab} \\
  \midrule
  Homomorphism 	& \cref{group-hom} 	& \cref{contr-group-hom} \\
  \midrule
  Composition	& \cref{comp} 		& \cref{contr-comp} \\
  		& \cref{thm:rec} 	& \cref{thm:contr-rec} \\
  \midrule
  Abelian noncyclic 	
  $2$-group 	& \cref{thm:pc} 	& \cref{thm:contr-pc} \\
  		& \cref{thm:non-cyclic-2} & \cref{thm:contr-non-cyclic-2} \\
  \bottomrule
\end{tabular}
\end{center}
\end{table}

\subsection{Finite field construction} \label{sec:contr-finite-field}
The constructions in this section are related to those in \cref{sec:finite-field}.

\begin{proposition}
  \label{contr-elem-ab-1}
  Let $p$ be prime, let $n$ be a positive integer, and let $\alpha$ be a primitive element of~$\GF(p^n)$. Then the additive form of the multiplication table for $\{1,\alpha,\alpha^2,\dots,\alpha^{n-1}\}$ is a $(\Z_{p}^{n}, n, 0)$ contracted difference matrix.
\end{proposition}

\begin{proof}
  The matrix corresponding to the multiplication table for $\{1,\alpha,\alpha^2,\dots,\alpha^{n-1}\}$ has $(i,j)$ entry $\alpha^{i+j}$ for $0 \le i,j \le n-1$.  We shall use Lemma~\ref{lem:contr-char}~(ii) to show that the additive form of this matrix is a $(\Z_p^n, n, 0)$ contracted difference matrix, where we regard $\Z_p^n$ as the additive group of $\GF(p^n)$.  Suppose that $0 = \sum_{i,j=0}^{n-1} a_i \alpha^{i+j} b_j$ for integers $a_i, b_j$ satisfying $-p < a_i, b_j < p$.
Then $\sum_{i=0}^{n-1} a_i \alpha^i = 0$ or $\sum_{j=0}^{n-1} b_j \alpha^j =0$, and since $\{1,\alpha,\alpha^2,\dots,\alpha^{n-1}\}$ is an integral basis for $\Z_p^n$ we conclude that either $a_i = 0$ for all $i$ or else $b_j = 0$ for all $j$.
\end{proof}

\begin{example}
  \label{ex:contr-field}
We use \cref{contr-elem-ab-1} to construct a $(\Z_2^3, 3, 0)$ contracted difference matrix. The shaded entries of \cref{ex:Z23}, constructed using a primitive element $\alpha$ of $\GF(2^3)$ satisfying $\alpha^3+\alpha+1 = 0$, comprise the additive form of the multiplication table for $\{1,\alpha,\alpha^2\}$ and so are a $(\Z_2^3, 3, 0)$ contracted difference matrix.
\end{example}

The next result extends the construction of \cref{contr-elem-ab-1} to give examples with $s > 0$.

\begin{lemma}
  \label{contr-group-hom}
  Let $p$ be prime, and let $G$ and $H$ be abelian groups of orders $p^{n+u}$ and $p^{n}$, respectively. Suppose that $\phi : G \to H$ is a surjective homomorphism and that $M=(m_{ij})$ is a $(G, k, s)$ contracted difference matrix. Then $\phi(M) = (\phi(m_{ij}))$ is an $(H, k, u + s)$ contracted difference matrix.
\end{lemma}

\begin{proof}
Let $\ba = (a_i)$ be a nonzero row vector of length $k$, where each $a_i$ is an integer satisfying $-p < a_i < p$. By Lemma~\ref{lem:contr-char}~(i), we are given that $\big \{\ba M \bc^\intercal : \bc \in \Z_p^{n+u+s}\big\}$ contains each element of $G$ exactly $p^s$ times and are required to prove that $\big \{\ba \phi(M) \bc^\intercal : \bc \in \Z_p^{n+u+s}\big\}$ contains each element of $H$ exactly $p^u p^s$ times. This follows from the First Isomorphism Theorem, because $|\Ker(\phi)| = |G|/|H| = p^u$.
\end{proof}

\begin{corollary}
  \label{contr-elem-ab}
  Let $p$ be prime, and let $k, n \ge 1$ and $s \ge 0$ be integers. Then there exists a $(\Z_{p}^{n}, k, s)$ contracted difference matrix if and only if $k \le n+s$.
\end{corollary}
\begin{proof}
The condition $k \le n+s$ is necessary, by \cref{thm:jungnickel-nonexistence} and \cref{defn:cdm}. To show existence for $k < n+s$, delete $n+s-k$ of the rows of the contracted difference matrix for $k= n+s$. 
It remains to construct a $(\Z_p^n,n+s,s)$ contracted difference matrix.
By \cref{contr-elem-ab-1}, there exists a $(\Z_p^{n+s},n+s,0)$ contracted difference matrix. Apply \cref{contr-group-hom} using a surjective homomorphism $\phi: \Z_p^{n+s} \to \Z_p^n$.
\end{proof}

\begin{example}
  We use \cref{contr-group-hom} to construct a $(\Z_2^2,4,2)$ contracted difference matrix.
  Firstly construct a $(\Z_{2}^{4}, 4, 0)$ contracted difference matrix according to \cref{contr-elem-ab-1}, using a primitive element $\alpha$ of $\GF(2^4)$ satisfying $\alpha^{4} + \alpha + 1 = 0$:
  \begin{equation*}
    \kbordermatrix{
      \cdot & 1    & \alpha & \alpha^2 & \alpha^3 \\
      1     & 0001 & 0010 & 0100  & 1000  \\
   \alpha   & 0010 & 0100 & 1000  & 0011  \\
   \alpha^2 & 0100 & 1000 & 0011  & 0110  \\
   \alpha^3 & 1000 & 0011 & 0110  & 1100
    }.
  \end{equation*}
Now apply the canonical homomorphism from $\Z_2^4$ to $\Z_2^2$ that removes the last two components of each element of $\Z_2^4$, giving the $(\Z_2^2, 4, 2)$ contracted difference matrix
  \begin{equation*}
    \begin{bmatrix}
      00 & 00 & 01 & 10 \\
      00 & 01 & 10 & 00 \\
      01 & 10 & 00 & 01 \\
      10 & 00 & 01 & 11
    \end{bmatrix}.
  \end{equation*}
\end{example}

\subsection{Composition construction}
The constructions in this section are related to those in \cref{sec:composition}.

\begin{theorem}
  \label{contr-comp}
  Let $G$ be an abelian $p$-group and let $H$ be a subgroup of~$G$. Suppose that $L$ is an $(H, k, s)$ contracted difference matrix and that $(m_{ij}+H)$ is a $(G/H, k, t)$ contracted difference matrix, and let $M = (m_{ij})$. Then the matrix $\begin{bmatrix}L\mid M\end{bmatrix}$ is a $(G, k, s + t)$ contracted difference matrix.
\end{theorem}

\begin{proof}
Let $|H| = p^n$ and $|G| = p^{n+u}$.
Let $\ba = (a_i)$ be a nonzero row vector of length $k$, where each $a_i$ is an integer satisfying $-p < a_i < p$. By \cref{lem:contr-char}~(i), we are given that 
$\big \{\ba L \bc^\intercal : \bc \in \Z_p^{n+s}\big\}$ contains each element of $H$ exactly $p^s$ times and that
$\big \{\ba (m_{ij}+H) \bd^\intercal : \bd \in \Z_p^{u+t}\big\}$ contains each element of $G/H$ exactly $p^t$ times.
Since each element of $G$ can be written uniquely as the sum of an element of $H$ and a coset representative of $H$ in $G$, it follows that
$\big \{\ba [L \mid M] (\bc,\bd)^\intercal : (\bc,\bd) \in \Z_p^{n+s+u+t}\big\}$ contains each element of $G$ exactly $p^sp^t$ times. 
The result follows by \cref{lem:contr-char}~(i).
\end{proof}

\begin{example}
  We use \cref{contr-comp} to construct a $(G,2,0)$ contracted difference matrix for $G = \Z_{9} \times \Z_{3} \times \Z_{3}$. Let $H = \langle 300, 010 \rangle$, so that both $H$ and $G/H$ are isomorphic to $\Z_{3} \times \Z_{3}$. Apply \cref{contr-elem-ab-1}, using a primitive element $\alpha$ of $\GF(3^2)$ satisfying $\alpha^{2} + \alpha + 2=0$, to construct the $(H, 2, 0)$ and $(G / H, 2, 0)$ contracted difference matrices
  \begin{equation*}
    \begin{bmatrix}
      010 & 300 \\
      300 & 610
    \end{bmatrix}
    \qquad\text{and}\qquad
    \begin{bmatrix}
      001 + H & 100 + H \\
      100 + H & 201 + H
    \end{bmatrix}.
  \end{equation*}
  By \cref{contr-comp}, a $(G, 2, 0)$ contracted difference matrix is
  \begin{equation*}
    \begin{bmatrix}
      010 & 300 & 001 & 100 \\
      300 & 610 & 100 & 201
    \end{bmatrix}.
  \end{equation*}
\end{example}

The proof of the next result follows that of \cref{thm:rec}, by replacing each quoted result for a difference matrix by the corresponding result for a contracted difference matrix according to \cref{tab:related-constructions}.
\begin{corollary}
  \label{thm:contr-rec}
  Let $p$ be prime, and let $G$ be an abelian group of order $p^{n}$ and exponent~$p^{e}$. Then there exists a $(G, \floor{n/e}, 0)$ contracted difference matrix.
\end{corollary}

\begin{example}
  We illustrate the proof of \cref{thm:contr-rec} with $p=2$, $n=11$, $e=3$ to construct a $(G,3,0)$ contracted difference matrix for $G = \Z_8 \times \Z_8 \times \Z_4 \times \Z_4 \times \Z_2$.
Following \cref{prop:buratti-chain}, we set $G=G_0$ and choose successive subgroups $G_1$, $G_2$, \dots so that each of the $\floor{n/e} =3$ largest direct factors of $G_{i-1}$ is reduced by a factor of $2$ in $G_i$, and determine $G_r$ to be the first resulting subgroup that is elementary abelian:
\begin{align*}
G_0 &=      \Z_8 \times \Z_8 \times \Z_4 \times \Z_4 \times \Z_2, \\
\langle 20000, 02000, 00200, 00010, 00001 \rangle = G_1 & \cong \Z_4 \times \Z_4 \times \Z_2 \times \Z_4 \times \Z_2, \\
\langle 40000, 04000, 00200, 00020, 00001 \rangle = G_2 & \cong \Z_2 \times \Z_2 \times \Z_2 \times \Z_2 \times \Z_2.
\end{align*}
This gives $r=2$ and $G_0/G_1 \cong \Z_2^3$ and $G_1/G_2 \cong \Z_2^3$ and $G_2 \cong \Z_2^5$. 

Use \cref{ex:contr-field} to construct the $(G_0/G_1,3,0)$ and $(G_1/G_2,3,0)$ contracted difference matrices
\[
    \begin{bmatrix}
      00100 + G_1 & 01000 + G_1 & 10000 + G_1 \\
      01000 + G_1 & 10000 + G_1 & 01100 + G_1 \\
      10000 + G_1 & 01100 + G_1 & 11000 + G_1 \\
    \end{bmatrix}
\quad \mbox{and} \quad
    \begin{bmatrix}
      00010 + G_2 & 02000 + G_2 & 20000 + G_2 \\
      02000 + G_2 & 20000 + G_2 & 02010 + G_2 \\
      20000 + G_2 & 02010 + G_2 & 22000 + G_2
    \end{bmatrix}.
\]
Apply \cref{contr-elem-ab-1}, using a primitive element $\alpha$ of $\GF(2^5)$ satisfying $\alpha^5+\alpha^2+1 = 0$, to construct a $(G_2,5,0)$ contracted difference matrix and then delete the last two rows to leave a $(G_2,3,0)$ contracted difference matrix
\[
    \begin{bmatrix}
      00001 & 00020 & 00200 & 04000 & 40000 \\
      00020 & 00200 & 04000 & 40000 & 00201 \\
      00200 & 04000 & 40000 & 00201 & 04020
    \end{bmatrix}.
\]
Now use \cref{contr-comp} to combine the $(G_2,3,0)$ and $(G_1/G_2,3,0)$ contracted difference matrices to give a $(G_1,3,0)$ contracted difference matrix; then combine this with the $(G_0/G_1,3,0)$ contracted difference matrix to give a $(G_0,3,0)$ contracted difference matrix
\[
    \begin{bmatrix}
      00001 & 00020 & 00200 & 04000 & 40000 & 00010 & 02000 & 20000 & 00100 & 01000 & 10000 \\
      00020 & 00200 & 04000 & 40000 & 00201 & 02000 & 20000 & 02010 & 01000 & 10000 & 01100 \\
      00200 & 04000 & 40000 & 00201 & 04020 & 20000 & 02010 & 22000 & 10000 & 01100 & 11000
    \end{bmatrix}.
\]
\end{example}

\subsection{Abelian noncyclic $2$-group construction}

The constructions in this section are related to those in \cref{sec:abelian-noncyclic}.

\begin{theorem}
 \label{thm:contr-pc}
  Let $e$ be a positive integer. Then there exists a $(\Z_{2^{e}} \times \Z_{2}, 2, 0)$ contracted difference matrix.
\end{theorem}

\begin{proof}
We use Lemma~\ref{lem:contr-char}~(ii) to show that $2 \times (e+1)$ matrix 
  \begin{equation*}
    M = \begin{bmatrix}
      (0,1) & (1,0) & (2,0) & (4,0) & (8,0) & \cdots & (2^{e-1},0) \\
      (2^{e-1},1) & (0,1) & (1,0) & (2,0) & (4,0) & \cdots & (2^{e-2},0)
    \end{bmatrix}
  \end{equation*}
is a $(\Z_{2^{e}} \times \Z_{2}, 2, 0)$ contracted difference matrix.
Suppose that $\ba M \bb^\intercal = (0,0)$, where 
$\ba = \begin{bmatrix} a_1 & a_2 \end{bmatrix}$ and
$\bb = \begin{bmatrix} b_e & b_0 & b_1 & b_2 & \dots & b_{e-1}\end{bmatrix}$,
and that each $a_i, b_j \in \{-1,0,1\}$ and $(a_1, a_2) \ne (0,0)$. It is sufficient to show that each $b_j$ is~0.

Expand the equation $\ba M \bb^\intercal = (0,0)$ to give
\begin{align}
2^{e-1}a_2b_e + a_1b_0 + (2a_1+a_2)(b_1+2b_2+4b_3+\dots+2^{e-2}b_{e-1}) & \equiv 0 \pmod{2^e} \label{eqn:M1} \\
(a_1+a_2)b_e + a_2b_0 &\equiv 0 \pmod{2}. \label{eqn:M2}
\end{align}
\begin{description}
\item[Case 1: $a_2=0$.]
 By \eqref{eqn:M1},
\[
a_1(b_0+2b_1+4b_2+\dots+2^{e-1}b_{e-1}) \equiv 0 \pmod{2^e}. 
\]
Since $(a_1,a_2)\ne(0,0)$ we have $a_1 \ne 0$, and therefore $b_0 = b_1 = \dots = b_{e-1} = 0$. Then \eqref{eqn:M2} gives $a_1 b_e \equiv 0 \pmod{2}$, so that $b_e = 0$.

\item[Case 2: $a_2 \in \{-1,1\}$.]
We firstly note that $a_1 b_0=0$: if $a_1=0$ this is immediate, otherwise we have $a_1 + a_2 \equiv 0 \pmod{2}$ and then from \eqref{eqn:M2} we conclude that $b_0=0$. Now substitute $a_1 b_0=0$ in \eqref{eqn:M1} to give
\[
(2a_1+a_2)(b_1+2b_2+4b_3+\dots+2^{e-2}b_{e-1}+2^{e-1}b_e) \equiv 0 \pmod{2^e}.
\]
Since $2a_1+a_2 \in \{-3,-1,1,3\}$, this implies that
\[
b_1+2b_2+4b_3+\dots+2^{e-2}b_{e-1}+2^{e-1}b_e \equiv 0 \pmod{2^e}
\]
and so $b_1 = b_2 = \dots = b_e = 0$. Then from \eqref{eqn:M2} we have $b_0=0$.\qedhere
\end{description}
\end{proof}

We remark that \cref{thm:contr-pc} implies \cref{thm:pc}, and yet relies on a considerably simpler construction.

The proof of the next result follows that of \cref{thm:non-cyclic-2}, by replacing each quoted result for a difference matrix by the corresponding result for a contracted difference matrix according to \cref{tab:related-constructions}.

\begin{corollary}
  \label{thm:contr-non-cyclic-2}
  Let $G$ be an abelian noncyclic $2$-group. Then there exists a $(G, 2, 0)$ contracted difference matrix.
\end{corollary}

\section{Further examples of contracted difference matrices} 
\label{sec:new-contr-diff}

In this section we present a $(G,3,0)$ contracted difference matrix in each of four abelian 2-groups~$G$, and derive several consequences. These examples cannot be obtained from the constructions given in \cref{sec:contr-diff-matrices}, but were instead found by computer search. 

A principal advantage of searching for a contracted difference matrix is that an exhaustive search can be feasible even though an exhaustive search for the corresponding size of difference matrix is not, because only an exponentially smaller number of matrices need be considered: in the present case, when $G$ is an abelian group of order~$2^n$, a $(G,3,0)$ contracted difference matrix has $3n$ entries in~$G$ whereas a $(G,2^3,1)$ difference matrix has $2^{n+3}$ entries in~$G$. Even when neither exhaustive search is feasible, a random search for a $(G,3,0)$ contracted difference matrix appears experimentally to be successful far more often than a random search for a $(G,2^3,1)$ difference matrix.
A further advantage of searching for a contracted difference matrix is that \cref{lem:contr-char} allows us to check a candidate contracted difference matrix efficiently, without the need to expand the matrix and check the differences between all row pairs explicitly.

\begin{example}\label{ex:3new}
The following $(G,3,0)$ contracted difference matrices were found by computer search:
\begin{align*}
  \begin{bmatrix}
    001 & 010 & 100 & 200 \\
    010 & 201 & 001 & 100 \\
    211 & 100 & 201 & 210
  \end{bmatrix} 
  &\mbox{ over $G=\Z_4 \times \Z_2 \times \Z_2$,} \\
  \begin{bmatrix}
    001 & 010 & 020 & 100 & 200 \\
    021 & 001 & 211 & 010 & 100 \\
    220 & 101 & 200 & 210 & 320
  \end{bmatrix} 
  &\mbox{ over $G=\Z_4 \times \Z_4 \times \Z_2$,} \\
  \begin{bmatrix}
    0001 & 0010 & 0100 & 1000 & 2000 \\
    0010 & 0100 & 2001 & 0001 & 1000 \\
    1001 & 2110 & 0001 & 0010 & 0101
  \end{bmatrix} 
  &\mbox{ over $G=\Z_4 \times \Z_2 \times \Z_2 \times \Z_2$,} \\
  \begin{bmatrix}
    4101 & 2000 & 1100 & 0010 & 7111 & 0101 \\
    4010 & 0111 & 2011 & 4001 & 7001 & 0001 \\
    4000 & 7110 & 5100 & 0011 & 5010 & 2011
  \end{bmatrix}
  &\mbox{ over $G=\Z_8 \times \Z_2 \times \Z_2 \times \Z_2$.}
\end{align*}
\end{example}

\cref{small-2-grp-info} combines the contracted difference matrices of \cref{ex:3new} with the constructive results of Corollaries~\ref{thm:contr-rec} and~\ref{thm:contr-non-cyclic-2} to show the largest number of rows $k$ for which a $(G,k,0)$ contracted difference matrix is known to exist, for all abelian $2$-groups of order at most~64. The nonexistence results of Theorems~\ref{thm:jungnickel-nonexistence} and~\ref{thm:drake-nonexistence}, together with exhaustive search results, are used to indicate when the displayed value of $k$ is known to be the maximum possible.

\begin{table}
\begin{center}
\begin{threeparttable}
  \caption{The largest number of rows $k$ for which a $(G,k,0)$ contracted difference matrix is known to exist, for each abelian $2$-group $G$ of order at most~$64$. An example matrix for each of these groups is given in Appendix~\ref{sec:list-contr-diff}.
}
  \label{small-2-grp-info} 
\begin{tabular}{l|c|l|l}
  \toprule
  Group $G$ & \# rows $k$ & Source & Maximum possible $k$? \\
  \midrule
  $\Z_{2}$ & 1 & trivial & yes (\cref{thm:drake-nonexistence}) \\
  \midrule
  $\Z_{2} \times \Z_{2}$ & 2 & \cref{thm:contr-rec} & yes (\cref{thm:jungnickel-nonexistence}) \\
  $\Z_{4}$ & 1 & trivial & yes (\cref{thm:drake-nonexistence}) \\
  \midrule
  $\Z_{2} \times \Z_{2} \times \Z_{2}$ & 3 & \cref{thm:contr-rec} & yes (\cref{thm:jungnickel-nonexistence}) \\
  $\Z_{4} \times \Z_{2}$ & 2 & \cref{thm:contr-non-cyclic-2} & yes (computer search) \\
  $\Z_{8}$ & 1 & trivial & yes (\cref{thm:drake-nonexistence}) \\
  \midrule
  $\Z_{2} \times \Z_{2} \times \Z_{2} \times \Z_{2}$ & 4 & \cref{thm:contr-rec} & yes (\cref{thm:jungnickel-nonexistence}) \\
  $\Z_{4} \times \Z_{2} \times \Z_{2}$ & 3 & computer search & yes (computer search) \\
  $\Z_{4} \times \Z_{4}$ & 2 & \cref{thm:contr-non-cyclic-2} & yes (computer search) \\
  $\Z_{8} \times \Z_{2}$ & 2 & \cref{thm:contr-non-cyclic-2} & yes (computer search) \\
  $\Z_{16}$ & 1 & trivial & yes (\cref{thm:drake-nonexistence}) \\
  \midrule
  $\Z_{2} \times \Z_{2} \times \Z_{2} \times \Z_{2} \times \Z_{2}$ & 5 & \cref{thm:contr-rec} & yes (\cref{thm:jungnickel-nonexistence}) \\
  $\Z_{4} \times \Z_{2} \times \Z_{2} \times \Z_{2}$ & 3 & computer search & unknown \\
  $\Z_{4} \times \Z_{4} \times \Z_{2}$ & 3 & computer search & unknown$^*$ \\
  $\Z_{8} \times \Z_{2} \times \Z_{2}$ & 2 & \cref{thm:contr-non-cyclic-2} & unknown$^*$ \\
  $\Z_{8} \times \Z_{4}$ & 2 & \cref{thm:contr-non-cyclic-2} & unknown$^*$ \\
  $\Z_{16} \times \Z_{2}$ & 2 & \cref{thm:contr-non-cyclic-2} & unknown$^*$ \\
  $\Z_{32}$ & 1 & trivial & yes (\cref{thm:drake-nonexistence}) \\
  \midrule
  $\Z_{2} \times \Z_{2} \times \Z_{2} \times \Z_{2} \times \Z_{2} \times \Z_{2}$ & 6 & \cref{thm:contr-rec} & yes (\cref{thm:jungnickel-nonexistence}) \\
  $\Z_{4} \times \Z_{2} \times \Z_{2} \times \Z_{2} \times \Z_{2}$ & 3$^\dagger$ & \cref{thm:contr-rec} & unknown \\
  $\Z_{4} \times \Z_{4} \times \Z_{2} \times \Z_{2}$ & 3$^\dagger$ & \cref{thm:contr-rec} & unknown \\
  $\Z_{4} \times \Z_{4} \times \Z_{4}$ & 3 & \cref{thm:contr-rec} & unknown \\
  $\Z_{8} \times \Z_{2} \times \Z_{2} \times \Z_{2}$ & 3 & computer search & unknown \\
  $\Z_{8} \times \Z_{4} \times \Z_{2}$ & 2$^\dagger$ & \cref{thm:contr-non-cyclic-2} & unknown \\
  $\Z_{8} \times \Z_{8}$ & 2 & \cref{thm:contr-non-cyclic-2} & unknown \\
  $\Z_{16} \times \Z_{2} \times \Z_{2}$ & 2 & \cref{thm:contr-non-cyclic-2} & unknown \\
  $\Z_{16} \times \Z_{4}$ & 2 & \cref{thm:contr-non-cyclic-2} & unknown \\
  $\Z_{32} \times \Z_{2}$ & 2 & \cref{thm:contr-non-cyclic-2} & unknown \\
  $\Z_{64}$ & 1 & trivial & yes (\cref{thm:drake-nonexistence}) \\
  \bottomrule
\end{tabular}
\begin{tablenotes}
\item
$^*$ Known by exhaustive search to be the maximum possible $k$ when one of the rows of the contracted difference matrix takes its lexicographically first feasible value (the $2$-expansion of this row consisting of a row of only $0_G$ and a row containing every element of~$G$), for example $\begin{bmatrix} 001 & 010 & 020 & 100 & 200 \end{bmatrix}$ when $G = \Z_{4} \times \Z_{4} \times \Z_{2}$.
\item
$^\dagger$ \cref{q:rank-r} asks: can this number be increased by one? 
\item
\end{tablenotes}
\end{threeparttable}
\end{center}
\end{table}

\subsection{A new infinite family of $(G,3,0)$ contracted difference matrices} 
\label{sec:new-family}
We can use the four examples given in \cref{ex:3new} to construct a $(G,3,0)$ contracted difference matrix for infinitely many abelian 2-groups~$G$ that are not handled by the methods of Section~\ref{sec:contr-diff-matrices}. By \cref{contr-comp} it is sufficient to find a chain of subgroups $G_i$ of $G$ satisfying
\[
G = G_0 \supset G_1 \supset \dots \supset G_r \supset G_{r+1} = \{0_G\}
\]
such that there exists a $(G_{i-1}/G_i,3,0)$ contracted difference matrix for each $i$ satisfying \mbox{$1 \le i \le r+1$}. By \cref{contr-elem-ab-1} and \cref{ex:3new}, this is possible in particular if the quotient group $G_{i-1}/G_i$ is isomorphic to $\Z_2^r$ for some $r \ge 3$ or to one of the groups in the set
\[
  \{\Z_4 \times \Z_2 \times \Z_2, \, \, \Z_4 \times \Z_4 \times \Z_2, \, \, \Z_4 \times \Z_2 \times \Z_2 \times \Z_2, \, \, \Z_8 \times \Z_2 \times \Z_2 \times \Z_2 \},
\]
for each $i$. 
For example, to construct a $(G,3,0)$ contracted difference matrix for $G = \Z_{256} \times \Z_{32} \times \Z_{16} \times \Z_4 \times \Z_2$, we can use the subgroup chain
$G = G_0$,
$G_1 \cong \Z_{32} \times \Z_{16} \times \Z_{16} \times \Z_2$,
$G_2 \cong \Z_{16} \times \Z_4 \times \Z_4 \times \Z_2$,
$G_3 \cong \Z_2 \times \Z_2 \times \Z_2$,
$G_4 = \{0_G\}$.
It does not seem straightforward to determine the set of all such groups $G$ explicitly, but we can show the existence of a large set of such groups by a straightforward induction.

\begin{theorem}
  \label{thm:contr-3-rows-partial}
  Let $n \ge 3$ and $e \le n/2$ be positive integers. Then for at least one abelian group~$G$ of order $2^n$ and exponent $2^e$ there exists a $(G, 3, 0)$ contracted difference matrix.
\end{theorem}

\begin{proof}
  The proof is by induction on $n \geq 3$, using base cases $n = 3, 4, 5, 6$. The base cases are provided by entries of \cref{small-2-grp-info}.
Now assume all cases up to $n-4 \ge 3$ are true.
If $e=1$ or $2$, then case $n$ is true by \cref{thm:contr-rec} because $\floor{n/e} \ge 3$. Otherwise $e \ge 3$, and then by the inductive hypothesis there is a group $H$ of order $2^{n-4}$ and exponent $2^{e-2}$ for which there exists a $(H,3,0)$ contracted difference matrix. Take $G$ to be an abelian group of order $2^n$ and exponent $2^e$ such that $G/H \cong \Z_4\times\Z_2\times\Z_2$.  By \cref{ex:3new}, there exists a $(G/H,3,0)$ contracted difference matrix. Apply \cref{contr-comp} to produce a $(G,3,0)$ contracted difference matrix, proving case $n$ and completing the induction.
\end{proof}

\begin{example}
We construct a $(G, 3, 0)$ contracted difference matrix for $G = \Z_{16} \times \Z_{8} \times \Z_{4}$ (whereas \cref{thm:contr-rec} gives only a $(G,2,0)$ contracted difference matrix). Choose $G_1 \cong \Z_4 \times \Z_4 \times \Z_2$ to be a subgroup of $G$ such that $G/G_1 \cong \Z_4 \times \Z_2 \times \Z_2$. Using the $(Z_4\times\Z_4\times\Z_2,3,0)$ and $(Z_4\times\Z_2\times\Z_2,3,0)$ contracted difference matrices of \cref{ex:3new} in \cref{contr-comp} we then obtain the following $(G,3,0)$ contracted difference matrix:
  \begin{equation*}
    \begin{bmatrix}
    002 & 020 & 040 & 400 & 800    & 001 & 010 & 100 & 200 \\
    042 & 002 & 822 & 020 & 400    & 010 & 201 & 001 & 100 \\
    840 & 402 & 800 & 820 & (12)40 & 211 & 100 & 201 & 210
    \end{bmatrix}
  \end{equation*}
\end{example}

\cref{largest-known-size} shows how the results of \cref{thm:contr-3-rows-partial} improve those of Corollaries~\ref{thm:contr-rec} and~\ref{thm:contr-non-cyclic-2}.

\newlength\celldim
\newlength\fontheight
\newlength\extraheight
\newcommand\nl{\tabularnewline}
\newcolumntype{Z}{ @{} >{\centering$} p{\celldim} <{$}@{} }
\newcolumntype{S}{
  @{}
  >{\centering \rule[-0.5\extraheight]{0pt}{\fontheight + \extraheight}%
    \begin{minipage}{\celldim}\centering$}
    p{\celldim}
    <{$\end{minipage}} 
  @{} }

\setlength\celldim{1.8em}%
\settoheight\fontheight{A}%
\setlength\extraheight{\celldim - \fontheight}%

\begin{figure}
  \centering
  \begin{equation*}
    \setlength\arraycolsep{0pt}
    \begin{array}{SZ|ZZZZZZZZZZZZZZZ}
      \multirow{8}{*}{\vspace{-1.5cm}$e$} & 8 & & & & & & & & 1 & \cellcolor{thm2} 2 & \cellcolor{thm2} 2 & \cellcolor{thm2} 2 & \cellcolor{thm2} \tm{thm2}2 & \cellcolor{thm2} 2 & \cellcolor{thm2} 2 \nl
      & 7 & & & & & & & 1 & \cellcolor{thm2} 2 & \cellcolor{thm2} 2 & \cellcolor{thm2} 2 & \cellcolor{thm2} 2 & \cellcolor{thm2} 2 & \cellcolor{thm2} 2 & \cellcolor{new} 3 \nl
      & 6 & & & & & & 1 & \cellcolor{thm2} 2 & \cellcolor{thm2} 2 & \cellcolor{thm2} 2 & \cellcolor{thm2} 2 & \cellcolor{thm2} 2 & \cellcolor{new} 3 & \cellcolor{new} 3 & \cellcolor{new} 3\tm{new} \nl
      & 5 & & & & & 1 & \cellcolor{thm2} 2 & \cellcolor{thm2} 2 & \cellcolor{thm2} 2 & \cellcolor{thm2} 2 & \cellcolor{new} 3 & \cellcolor{new} 3 & \cellcolor{new} 3 & \cellcolor{new} 3 & \cellcolor{new} 3 \nl
      & 4 & & & & 1 & \cellcolor{thm2} 2 & \cellcolor{thm2} 2 & \cellcolor{thm2} 2 & \cellcolor{new} 3 & \cellcolor{new} 3 & \cellcolor{new} 3 & \cellcolor{new} 3 & \cellcolor{thm1} 3 & \cellcolor{thm1} 3 & \cellcolor{thm1} 3 \nl
      & 3 & & & 1 & \cellcolor{thm2} 2 & \cellcolor{thm2} 2 & \cellcolor{base} \textcolor{basetext}{\tm{b1}3} & \cellcolor{new} 3 & \cellcolor{new} 3 & \cellcolor{thm1} 3 & \cellcolor{thm1} 3 & \cellcolor{thm1} 3 & \cellcolor{thm1} 4 & \cellcolor{thm1} 4 & \cellcolor{thm1} 4\tm{thm1} \nl
      & 2 & & 1 & \cellcolor{thm2} 2 & \cellcolor{base} \textcolor{basetext}{\tm{b2}3} & \cellcolor{base} \textcolor{basetext}{\tm{b3}3} & \cellcolor{thm1} 3 & \cellcolor{thm1} 3 & \cellcolor{thm1} 4 & \cellcolor{thm1} 4 & \cellcolor{thm1} 5 & \cellcolor{thm1} 5 & \cellcolor{thm1} 6 & \cellcolor{thm1} 6 & \cellcolor{thm1} 7 \nl
      & 1 & 1 & \cellcolor{thm1} 2 & \cellcolor{thm1} 3 & \cellcolor{thm1} 4 & \cellcolor{thm1} 5 & \cellcolor{thm1} 6 & \cellcolor{thm1} 7 & \cellcolor{thm1} 8 & \cellcolor{thm1} 9 & \cellcolor{thm1} 10 & \cellcolor{thm1} 11 & \cellcolor{thm1} 12 & \cellcolor{thm1} 13 & \cellcolor{thm1} 14 \nl \cline{2-16}
      &   & 1 & 2 & 3 & 4 & 5 & 6 & 7 & 8 & 9 & 10 & 11 & 12 & 13 & 14 \nl
      \multicolumn{16}{c}{\hspace{1.3cm} n }
    \end{array}
  \end{equation*}
  \begin{tikzpicture}[overlay, remember picture]
    \node[left=0.5cm, above=3cm] at (pic cs:b2) (B) {\colorbox{base}{from \cref{ex:3new}}};
    \foreach \x in {b1, b2, b3}
    \draw[black!50, >=stealth, dashed, ->, shorten >=4mm] (B) -- ($(pic cs:\x) + (0.5ex,0.5em)$);
    \node[right=5mm, above=-2mm, rotate=-90] at (pic cs:thm1) (1) {\colorbox{thm1}{\cref{thm:contr-rec}}};
    \node[right=5mm, above=4mm, rotate=-90] at (pic cs:new) (3) {\colorbox{new}{~\cref{thm:contr-3-rows-partial}}};
    \node[right=-4mm, above=6mm] at (pic cs:thm2) (2) {\colorbox{thm2}{\cref{thm:contr-non-cyclic-2}}};
  \end{tikzpicture}
  \caption{The largest number of rows $k$ for which a $(G, k, 0)$ contracted difference matrix is known to exist for at least one abelian group $G$ of order~$2^n$ and exponent~$2^e$.}
  \label{largest-known-size}
\end{figure}

\subsection{New linking systems of difference sets of size~$7$}
\label{sec:link-syst-diff}
Question~1 of \cite[Section~6]{jedwab-li-simon-arxiv} asks whether there are examples of difference matrices over $2$-groups having more rows than those specified in Theorems~\ref{thm:rec} and~\ref{thm:non-cyclic-2}.
The difference matrices described in \cref{prop:lazebnik-thomason} provide such examples. By \cref{defn:cdm}, the four new $(G,3,0)$ contracted difference matrices described in \cref{sec:new-family} immediately give $(G,8,1)$ difference matrices that likewise cannot be obtained from those two theorems (giving, in particular, a second source for a $(\Z_4 \times \Z_2 \times \Z_2, 8, 1)$ difference matrix).

Moreover, we now describe how the four new contracted difference matrices of Section~\ref{sec:new-family} can be used to construct reduced linking systems of difference sets of size 7 in certain abelian $2$-groups $G$ of order $2^{2d+2}$ for which the largest previously known size was~3. To construct such a reduced linking system, it is sufficient by \cref{thm:jls} for $G$ to contain a subgroup $E$ isomorphic to $\Z_2^{d+1}$ such that there is a $(G/E,3,0)$ contracted difference matrix. We may therefore choose $H$ to be any group satisfying the quotient chain condition described in \cref{sec:new-family}, and then choose $G$ to be any abelian group containing a subgroup $E$ isomorphic to $\Z_2^{d+1}$ such that $G/E \cong H$. It again does not seem straightforward to determine the set of all such groups $G$ explicitly (in the absence of an explicit condition for all suitable groups $H$), but we can show the existence of a large set of such groups~$G$.

\begin{corollary}\label{cor:new-linking}
Let $d \ge 2$ and let $e$ be an integer satisfying $2 \le e \le \frac{d+3}{2}$. Then there is at least one abelian group $G$ of order $2^{2d+2}$, rank $d+1$, and exponent $2^e$, such that
there is a reduced linking system of $(v,k,\lambda,n)$-difference sets in $G$ of size~$7$, with $(v,k,\lambda,n)$ as given in~\eqref{eqn:hadamard}.
\end{corollary}

\begin{proof}
Applying \cref{thm:contr-3-rows-partial} with $n = d+1$ and $e$ replaced by $e-1$, there is at least one abelian group $H$ of order $2^{d+1}$ and exponent $2^{e-1}$ for which there exists an $(H,3,0)$ contracted difference matrix and therefore an $(H,8,1)$ difference matrix.
Since $H$ has rank at most $d+1$, there is therefore a group $G$ of order $2^{2d+2}$, rank $d+1$, and exponent $2^e$, containing a subgroup $E \cong \Z_{2}^{d+1}$ such that $G/E \cong H$. The result then follows from \cref{thm:jls}.
\end{proof}

The reduced linking systems of difference sets given by \cref{cor:jls}~(i) all have size 3, and those given by \cref{cor:jls}~(ii) (under the stated condition $2 \le e \le \frac{d+3}{2}$) have size~3 when $\frac{d+1}{e-1} < 3$. \cref{cor:new-linking} therefore provides reduced linking systems of difference sets, that are larger than those previously known, for all $d \ge 2$ and all~$e$ satisfying $\frac{d+4}{3} < e \le \frac{d+3}{2}$.

\begin{example}
The largest known size of a reduced linking system of $(256,120,56,64)$-difference sets in an abelian group of order $256$ and exponent $8$ was given as~$3$ in \cite[Table~3]{jedwab-li-simon-arxiv}. Using the above procedure, we can increase this size to $7$ for each of the groups 
\[
\Z_8 \times \Z_2^5, \quad \Z_8 \times \Z_4 \times \Z_2^3, \quad \Z_8 \times \Z_4^2 \times \Z_2 
\]
(but not for the group $\Z_8^2 \times \Z_2^2$) by choosing a subgroup $E$ isomorphic to $\Z_2^4$ such that $G/E \cong \Z_4 \times \Z_2 \times \Z_2$, and using the $(\Z_4 \times \Z_2 \times \Z_2,3,0)$ contracted difference matrix of \cref{ex:3new} to provide a $(\Z_4 \times \Z_2 \times \Z_2,8,1)$ difference matrix for use in \cref{thm:jls}.
\end{example}

We remark that additional reduced linking systems of difference sets of size~7 can be obtained by applying the composition construction of \cref{comp} to the $(Z_4 \times \Z_4, 8, 1)$ difference matrix of \cref{prop:lazebnik-thomason}~(ii) to produce a further infinite family of difference matrices with 8 rows in abelian 2-groups. This infinite family does not arise from the analogous construction for contracted difference matrices, because there is no $(\Z_4 \times \Z_4, 3, 0)$ contracted difference matrix (see \cref{small-2-grp-info}).

\section{Open Questions}
\label{sec:open-questions}

We conclude with three open questions.

The first question concerns how effective the concept of a contracted difference matrix is in explaining the existence pattern of difference matrices in abelian $p$-groups. It is motivated by comparing the results of \cref{elem-ab-1}, \cref{thm:rec}, \cref{thm:non-cyclic-2} for difference matrices with those of \cref{contr-elem-ab-1}, \cref{thm:contr-rec}, \cref{thm:contr-non-cyclic-2}, respectively, for contracted difference matrices.
\begin{question}
  \label{q:contr-implies-diff}
  For which primes $p$, abelian $p$-groups $G$, and integers $k \ge 1$ does there exist a $(G, p^k, 1)$ difference matrix but not a $(G, k, 0)$ contracted difference matrix?
\end{question}
\noindent
The only example currently known that satisfies the conditions of \cref{q:contr-implies-diff} is given by $p=2$ and $G = \Z_4 \times \Z_4$ and $k=3$, from \cref{prop:lazebnik-thomason}~(ii) and the exhaustive search result for $\Z_4 \times \Z_4$ in \cref{small-2-grp-info}. 

The second question concerns the largest number of rows of a difference matrix over an abelian $p$-group. 
\begin{question}
  \label{q:log}
  For which primes $p$ and abelian $p$-groups $G$ is the largest integer $m$ for which a $(G, m, 1)$ difference matrix exists not a power of~$p$?
\end{question}
\noindent
The only example currently known that satisfies the conditions of \cref{q:log} is given by $p=2$ and $G = \Z_8 \times \Z_2$ and $m=5$, from \cref{prop:lazebnik-thomason}~(i).
We also know that $p=3$ and $G=\Z_9 \times \Z_3$ and $k=2$ satisfy the conditions of at least one of Questions~\ref{q:contr-implies-diff} and~\ref{q:log}: by \cref{pan-chang-non-2} there exists a $(\Z_9 \times \Z_3, 4, 1)$ difference matrix, but exhaustive search shows there is no $(\Z_9 \times \Z_3, 2, 0)$ contracted difference matrix.

The third question follows from the observation that the currently known existence pattern for contracted difference matrices over $2$-groups of fixed order, as illustrated in \cref{small-2-grp-info}, appears to favor groups of smaller exponent and larger rank. 

\begin{question}
  \label{q:rank-r}
  Let $G$ be an abelian group of order $2^{n}$, exponent at most $2^{\frac{n}{r-1}}$, and rank at least~$r$. Does there exist a $(G, r, 0)$ contracted difference matrix?
\end{question}

A positive answer to \cref{q:rank-r} for the case $r=2$ is given by \cref{thm:contr-non-cyclic-2}; and a positive answer for the case $r=3$ would allow the words ``for at least one abelian group'' in \cref{thm:contr-3-rows-partial} to be replaced by ``for all abelian groups'', provided that a minimum rank of $3$ is specified.

\section*{Acknowledgements}
We are grateful to Patric {\"O}sterg{\aa}rd for pointing out the paper \cite{lazebnik-thomason}, and to Andrew Thomason for providing an example of a difference matrix attaining the bound given in each of the three parts of \cref{prop:lazebnik-thomason}.

\newpage

\begin{appendices}

\section{Contracted difference matrix examples}
\label{sec:list-contr-diff}

This appendix gives an example $(G,k,0)$ contracted difference matrix containing the largest known number of rows~$k$ (as stated in \cref{small-2-grp-info}), for each abelian $2$-group $G$ of order at most 64.

\begin{longtable}{>{$}l<{$}>{$}l<{$}}
  \label{tab:small-2-grp-examples} \\
  \toprule
  \text{Group} & \text{Matrix} \\
  \midrule \addlinespace
  \Z_{2} & 
  \begin{bmatrix}
    1
  \end{bmatrix} \\
  \addlinespace \midrule \addlinespace
  \Z_{2} \times \Z_{2} & 
  \begin{bmatrix}
    01 & 10 \\
    10 & 11
  \end{bmatrix} \\ \addlinespace
  \Z_{4} & 
  \begin{bmatrix}
    1 & 2
  \end{bmatrix} \\
  \addlinespace \midrule \addlinespace
  \Z_{2} \times \Z_{2} \times \Z_{2} & 
  \begin{bmatrix}
    001 & 010 & 100 \\
    010 & 100 & 011 \\
    100 & 011 & 110
  \end{bmatrix} \\ \addlinespace
  \Z_{4} \times \Z_{2} & 
  \begin{bmatrix}
    01 & 10 & 20 \\
    21 & 01 & 10
  \end{bmatrix} \\ \addlinespace
  \Z_{8} & 
  \begin{bmatrix}
    1 & 2 & 4
  \end{bmatrix} \\
  \addlinespace \midrule \addlinespace
  \Z_{2} \times \Z_{2} \times \Z_{2} \times \Z_{2} & 
  \begin{bmatrix}
    0001 & 0010 & 0100 & 1000 \\
    0010 & 0100 & 1000 & 0011 \\
    0100 & 1000 & 0011 & 0110 \\
    1000 & 0011 & 0110 & 1100
  \end{bmatrix} \\ \addlinespace
  \Z_{4} \times \Z_{2} \times \Z_{2} & 
  \begin{bmatrix}
    001 & 010 & 100 & 200 \\
    010 & 201 & 001 & 100 \\
    211 & 100 & 201 & 210
  \end{bmatrix} \\ \addlinespace
  \Z_{4} \times \Z_{4} & 
  \begin{bmatrix}
    01 & 10 & 02 & 20 \\
    10 & 11 & 20 & 22
  \end{bmatrix} \\ \addlinespace
  \Z_{8} \times \Z_{2} & 
  \begin{bmatrix}
    01 & 10 & 20 & 40 \\
    41 & 01 & 10 & 20
  \end{bmatrix} \\ \addlinespace
  \Z_{16} & 
  \begin{bmatrix}
    1 & 2 & 4 & 8
  \end{bmatrix} \\
  \addlinespace \midrule \addlinespace
  \Z_{2} \times \Z_{2} \times \Z_{2} \times \Z_{2} \times \Z_{2} & 
  \begin{bmatrix}
    00001 & 00010 & 00100 & 01000 & 10000 \\
    00010 & 00100 & 01000 & 10000 & 00101 \\
    00100 & 01000 & 10000 & 00101 & 01010 \\
    01000 & 10000 & 00101 & 01010 & 10100 \\
    10000 & 00101 & 01010 & 10100 & 01101
  \end{bmatrix} \\ \addlinespace
  \Z_{4} \times \Z_{2} \times \Z_{2} \times \Z_{2} & 
  \begin{bmatrix}
    0001 & 0010 & 0100 & 1000 & 2000 \\
    0010 & 0100 & 2001 & 0001 & 1000 \\
    1001 & 2110 & 0001 & 0010 & 0101
  \end{bmatrix} \\ \addlinespace
  \Z_{4} \times \Z_{4} \times \Z_{2} & 
  \begin{bmatrix}
    001 & 010 & 020 & 100 & 200 \\
    021 & 001 & 211 & 010 & 100 \\
    220 & 101 & 200 & 210 & 320
  \end{bmatrix}\\ \addlinespace
  \Z_{8} \times \Z_{2} \times \Z_{2} & 
  \begin{bmatrix}
    010 & 100 & 200 & 001 & 400 \\
    210 & 010 & 100 & 400 & 401
  \end{bmatrix} \\ \addlinespace
  \Z_{8} \times \Z_{4} & 
  \begin{bmatrix}
    01 & 10 & 20 & 02 & 40 \\
    21 & 01 & 10 & 40 & 42
  \end{bmatrix} \\ \addlinespace
  \Z_{16} \times \Z_{2} & 
  \begin{bmatrix}
    01 & 10 & 20 & 40 & 80 \\
    81 & 01 & 10 & 20 & 40
  \end{bmatrix} \\ \addlinespace
  \Z_{32} & 
  \begin{bmatrix}
    1 & 2 & 4 & 8 & (16)
  \end{bmatrix} \\
  \addlinespace \midrule \addlinespace
  \Z_{2} \times \Z_{2} \times \Z_{2} \times \Z_{2} \times \Z_{2} \times \Z_{2} &
  \setlength\arraycolsep{3.2pt}
  \begin{bmatrix}
    000001 & 000010 & 000100 & 001000 & 010000 & 100000 \\
    000010 & 000100 & 001000 & 010000 & 100000 & 000011 \\
    000100 & 001000 & 010000 & 100000 & 000011 & 000110 \\
    001000 & 010000 & 100000 & 000011 & 000110 & 001100 \\
    010000 & 100000 & 000011 & 000110 & 001100 & 011000 \\
    100000 & 000011 & 000110 & 001100 & 011000 & 110000
  \end{bmatrix} \\ \addlinespace
  \Z_{4} \times \Z_{2} \times \Z_{2} \times \Z_{2} \times \Z_{2} & 
  \setlength\arraycolsep{3.2pt}
  \begin{bmatrix}
    00100 & 01000 & 20000 & 00001 & 00010 & 10000 \\
    01000 & 20000 & 01100 & 00010 & 10000 & 00011 \\
    20000 & 01100 & 21000 & 10000 & 00011 & 10010
  \end{bmatrix} \\ \addlinespace
  \Z_{4} \times \Z_{4} \times \Z_{2} \times \Z_{2} & 
  \begin{bmatrix}
    0010 & 0200 & 2000 & 0001 & 0100 & 1000 \\
    0200 & 2000 & 0210 & 0100 & 1000 & 0101 \\
    2000 & 0210 & 2200 & 1000 & 0101 & 1100
  \end{bmatrix} \\ \addlinespace
  \Z_{4} \times \Z_{4} \times \Z_{4} & 
  \begin{bmatrix}
    002 & 020 & 200 & 001 & 010 & 100 \\
    020 & 200 & 022 & 010 & 100 & 011 \\
    200 & 022 & 220 & 100 & 011 & 110
  \end{bmatrix} \\ \addlinespace
  \Z_{8} \times \Z_{2} \times \Z_{2} \times \Z_{2} & 
  \begin{bmatrix}
    4101 & 2000 & 1100 & 0010 & 7111 & 0101 \\
    4010 & 0111 & 2011 & 4001 & 7001 & 0001 \\
    4000 & 7110 & 5100 & 0011 & 5010 & 2011
  \end{bmatrix} \\ \addlinespace
  \Z_{8} \times \Z_{4} \times \Z_{2} & 
  \begin{bmatrix}
    020 & 400 & 010 & 200 & 001 & 100 \\
    400 & 420 & 200 & 210 & 100 & 101
  \end{bmatrix} \\ \addlinespace
  \Z_{8} \times \Z_{8} & 
  \begin{bmatrix}
    04 & 40 & 02 & 20 & 01 & 10 \\
    40 & 44 & 20 & 22 & 10 & 11
  \end{bmatrix} \\ \addlinespace
  \Z_{16} \times \Z_{2} \times \Z_{2} & 
  \begin{bmatrix}
    010 & 800 & 001 & 100 & 200 & 400 \\
    800 & 810 & 401 & 001 & 100 & 200
  \end{bmatrix} \\ \addlinespace
  \Z_{16} \times \Z_{4} & 
  \begin{bmatrix}
    02 & 80 & 01 & 10 & 20 & 40 \\
    80 & 82 & 41 & 01 & 10 & 20
  \end{bmatrix} \\ \addlinespace
  \Z_{32} \times \Z_{2} & 
  \begin{bmatrix}
    01    & 10 & 20 & 40 & 80 & (16)0 \\
    (16)1 & 01 & 10 & 20 & 40 & 80
  \end{bmatrix} \\ \addlinespace
  \Z_{64} &
  \begin{bmatrix}
    1 & 2 & 4 & 8 & (16) & (32)
  \end{bmatrix} \\
  \addlinespace \bottomrule
\end{longtable}

\end{appendices}

\end{document}